\documentclass[reqno, 12pt]{amsart}
\usepackage[letterpaper,hmargin=1in,vmargin=1in]{geometry}
\usepackage{amsmath, amssymb, amsthm, verbatim, bbm, url} 
\usepackage{graphicx}
\usepackage{enumerate}

\newcommand \N {\mathbb{N}}
\newcommand \R {\mathbb{R}}

\newcommand \Z {\mathbb{Z}}
\newcommand \Sc {\mathcal{S}}

\newcommand \HH {\mathbb{H}}

\newcommand \RR {\mathbb{R}}

\newcommand \Oh {\mathcal{O}}

\newcommand \la {\langle}
\newcommand \ra {\rangle}

\newcommand \D {\partial}
\newcommand \eps {\varepsilon}

\newcommand \Def {\stackrel{\textrm{def}}=}

\newcommand\CI {C^\infty}
\DeclareMathOperator \re {Re}
\DeclareMathOperator \im {Im}

\DeclareMathOperator \supp {supp}

\DeclareMathOperator \WF {WF}
\DeclareMathOperator \WFh {WF_{\textit{h}}}
\DeclareMathOperator \Op {Op}
\DeclareMathOperator \Id {Id}

\newtheorem{lem}{Lemma}
\newtheorem{thm}{Theorem}
\newtheorem{prop}[lem]{Proposition}

\theoremstyle{definition}

\newtheorem{rem}[lem]{Remark}

\numberwithin{equation}{section}
\numberwithin{lem}{section}
\numberwithin{Defn}{section}
\numberwithin{thm}{section}

\title
{Propagation through trapped sets and semiclassical resolvent estimates}

\author[Kiril Datchev]
{Kiril Datchev}
\address{Department of Mathematics, Massachusetts Institute of Technology, Cambridge, MA
02139-4397, U.S.A.}
\email{datchev@math.mit.edu}
\author[Andr\'as Vasy]
{Andr\'as Vasy}
\address{Department of Mathematics, Stanford University, Stanford, CA
94305-2125, U.S.A.}
\email{andras@math.stanford.edu}
\thanks{The first author is partially supported by a National Science Foundation postdoctoral fellowship, and the second author is partially supported by the National Science Foundation under
grant DMS-0801226, and a Chambers Fellowship from Stanford University. The authors are grateful for the hospitality of the Mathematical Sciences Research Institute, where part of this research was carried out.}
\date{May 25, 2012}
\begin{document}

\begin{abstract}
Motivated by the study of resolvent estimates in the presence of trapping, we prove a semiclassical propagation theorem in a neighborhood of a compact invariant subset of the bicharacteristic flow which is isolated in a suitable sense. Examples include a global trapped set and a single isolated periodic trajectory. This is applied to obtain microlocal resolvent estimates with no loss compared to the nontrapping setting.
\end{abstract}

\maketitle

\section{Introduction}

In this paper we study the following phenomenon: losses in high energy, i.e. semiclassical, resolvent estimates caused by trapping are removed if one truncates the resolvent (microlocally) away from the trapped set. Such results go back to work of Burq \cite{Burq:Lower} and Cardoso and Vodev \cite{Cardoso-Vodev:Uniform}. Our result is based on a microlocal propagation estimate and is able to distinguish between different components of the trapped set. As an illustration, consider the following example:

Let $(X,g)$ be the catenoid or the hyperbolic cylinder, i.e. the quotient of the hyperbolic upper half plane by $\la z \mapsto 2z\ra$. Let $P = h^2\Delta_g - 1$. Let $R_h(\lambda) = (P - \lambda)^{-1}$. We are interested in behavior of this resolvent family when $\re \lambda = 0$, $\im \lambda \to 0^+$ (this corresponds to energy $1/h^2$ for the non-semiclassical $\Delta_g$).  It is well known that the limiting behavior of the resolvent is closely connected to dynamics of the geodesic flow on the energy surface, i.e. on the unit cosphere bundle. In this case the trapped, or nonwandering, set consists of two periodic orbits whose projections to $X$ are the same, see Figure~\ref{f:intro}. Denote these two orbits by $\Gamma^1$ and $\Gamma^2$, and denote by $\Gamma^1_\pm$ the set of $\rho \in S^*X$ such that the lifted geodesic through $\rho$ tends to $\Gamma^1$ as $t \to \mp \infty$, and define $\Gamma^2_\pm$ similarly. Let $u= R_h(\lambda)f$ with $\lambda$ as above. If $f$ is $\Oh(1)$, then $u$ is $\Oh(|\log h|\, h^{-1})$ by a result of Christianson \cite{c07,c08}. A consequence of our main result is that if in addition $f$ vanishes microlocally near $\Gamma^1$ but not near $\Gamma^2$, then $u$ is actually $\Oh(h^{-1})$ on $T^*X \setminus (\Gamma^1 \cup \Gamma^2_+)$. If we assume that $f$ vanishes microlocally near $\Gamma^2$ as well, then a result of Cardoso and Vodev \cite{Cardoso-Vodev:Uniform} (following earlier work of Burq \cite{Burq:Lower}) implies that $u$ is $\Oh(h^{-1})$ on $T^*X \setminus(\Gamma^1 \cup \Gamma^2)$. The novelty in this example is that we keep this improvement on $\Gamma^1_+$ even when $f$ is nontrivial on $\Gamma^2$. 

\begin{figure}
\includegraphics[width=16cm]{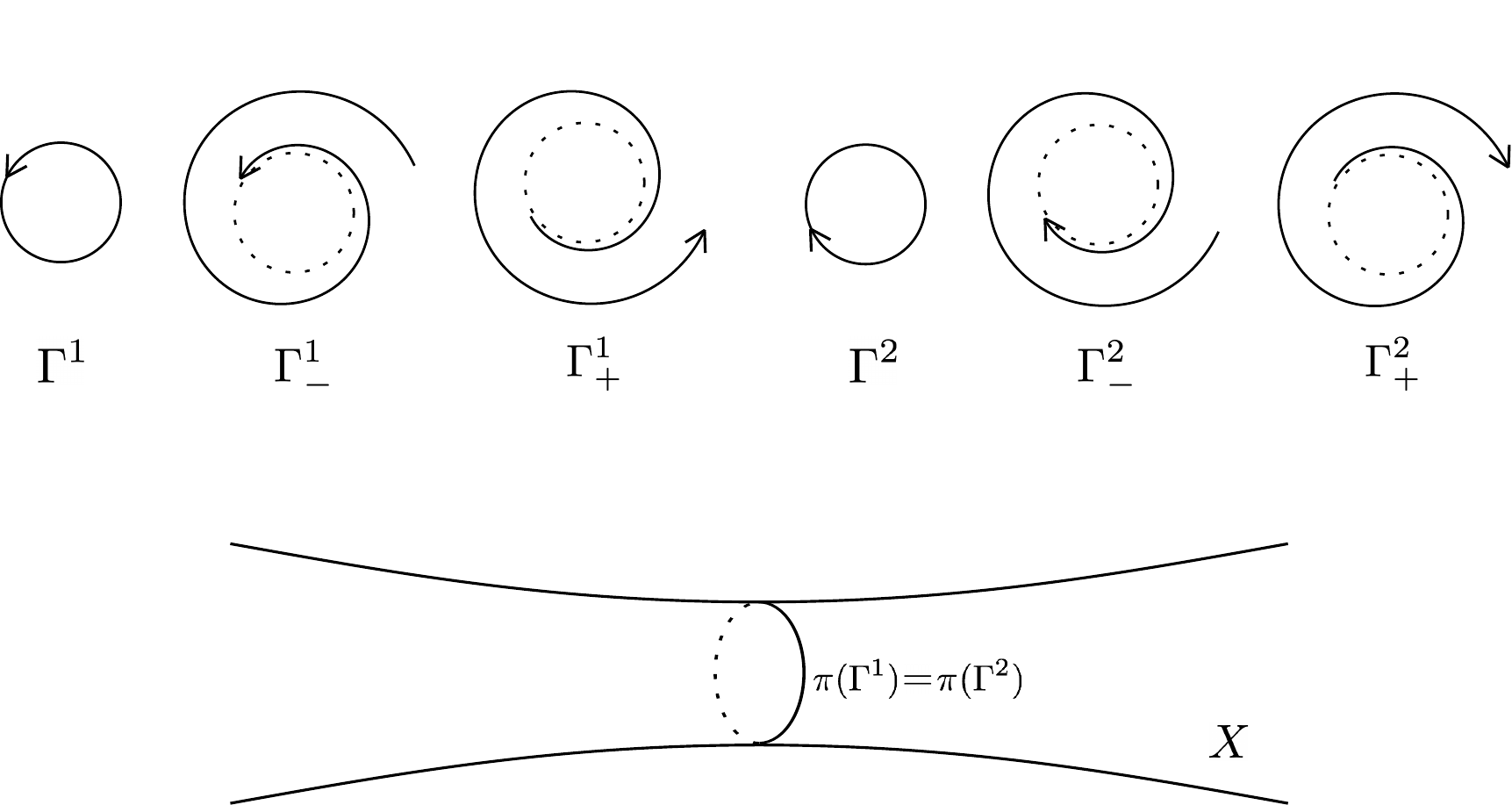}
\caption{The two closed orbits $\Gamma^1$ and $\Gamma^2$ are obtained by lifting the geodesic at the neck of the catenoid or hyperbolic cylinder to $S^*X$. The sets $\Gamma^j_\pm$, which by definition contain the $\Gamma^j$, each consist of the infinitely many trajectories spiraling towards $\Gamma^j$ as $t \to \mp \infty$. If $u = R_h(\lambda)f $ and $f$ is $\Oh(1)$ then $u$ is $\Oh(|\log h| \, h^{-1})$ globally by \cite{c07,c08}. If $f$ vanishes microlocally near $\Gamma^1\cup\Gamma^2$ then $u$ is actually $\Oh(h^{-1})$ off of $\Gamma^1 \cup \Gamma^2$ by \cite{Cardoso-Vodev:Uniform}. If $f$ vanishes microlocally only near $\Gamma^1$, we find that $u$ is actually $\Oh(h^{-1})$ off of $\Gamma^1 \cup \Gamma^2_+$.}\label{f:intro}
\end{figure}

More generally, let $(X,g)$ be a complete Riemannian manifold, $P=h^2\Delta_g+V-1$ a semiclassical Schr\"odinger operator, $V\in\CI(X;\RR)$ bounded, $h\in(0,1)$. We say that a bicharacteristic (by which we always mean a bicharacteristic in $\Sigma =p^{-1}(I)$ for some $I \subset \R$ compact) is \textit{backward nontrapped} if the flowout of any point on it is disjoint from any compact set for sufficiently negative time (this definition is generalized in \S\ref{s:defnot}). Suppose the resolvent family $R_h(\lambda)$ for $\lambda \in D \subset \{\re \lambda \in I, \im \lambda \ge -\Oh(h^{\infty})\}$, where $D$ is any subset, is polynomially bounded in $h$ over compact subsets of $T^*X$. This means that for any $a, b \in C_0^\infty(T^*X)$ there is $k \in \N$ such that $\|\Op(a)R_h(\lambda)\Op(b)\|_{L^2 \to L^2} \le h^{-k}$. Suppose further that $R_h(\lambda)$ is \textit{semiclassically outgoing} with a loss of $h^{-1}$ at backward nontrapped points in the following sense: if $u = R_h(\lambda)f$  and $\rho$ lies on a backward nontrapped bicharacteristic, and if $f$ is $\Oh(1)$ on the backward flowout of $\rho$, then $u$ is $\Oh(h^{-1})$ at $\rho$. Suppose also that $\tilde \Gamma$, the trapped set (the set of precompact bicharacteristics), is compact.

The following theorem generalizes the example at the beginning of the introduction:

\begin{thm}\label{t:cor} Let $(X,g)$, $P$ and $\lambda$ be as in the above paragraph. Let $a \in C_0^\infty(T^*X)$ have support disjoint from $\tilde \Gamma$, the trapped set. Let $b \in C_0^\infty(T^*X)$ have support disjoint from all connected components of $\tilde \Gamma$ intersecting the closure of the backward bicharacteristic flowout of $\supp a$.  Then nontrapping estimates hold:
\begin{equation}\label{e:ann}\|\Op(a) R_h(\lambda)\Op(b)\|_{L^2(X) \to L^2(X)}  \le C h^{-1},\end{equation}
\end{thm}
Here $\Op$ denotes the semiclassical quantization: see \S\ref{s:defnot}. Since the projection of the cotangent bundle to the base $\pi \colon T^*X \to X$ is a proper map when restricted to $\Sigma$, the condition that $a, b \in C_0^\infty(T^*X)$ can be weakened using microlocal elliptic regularity. Indeed, we may replace that condition with the condition that  $a, b \in C^\infty(T^*X)$ are bounded together with all derivatives, and that $\pi \supp a$ and $\pi \supp b$ are compact.

Note that if $X$ has suitable ends at infinity (for instance, asymptotically conic or hyperbolic), then the semiclassically outgoing assumption is satisfied (see \S\ref{s:outgoing} below),  we can use resolvent gluing to weaken the condition that $\pi \supp a$ and $\pi \supp b$ are compact to a decay condition, leading to the following theorem. 

\begin{thm}\label{t:first} Let $(X,g)$ be a complete Riemannian manifold which is either asymptotically conic or asymptotically hyperbolic and even in the sense of \S\ref{s:defnot}, let $\Delta_g$ be the nonnegative Laplace-Beltrami operator on $X$, let $V \in C_0^\infty(X)$, and fix $E>0$. Suppose that for any $\chi_0 \in C_0^\infty(X)$ there exist $C_0, k, h_0 > 0$ such that for any $\eps > 0$, $h \in (0,h_0]$ we have
\begin{equation}\label{e:polybd}
\|\chi_0(h^2\Delta_g + V - E - i\eps)^{-1}\chi_0\|_{L^2(X) \to L^2(X)} \le C_0 h^{-k}.
\end{equation}
Let $K_E \subset T^*X$ be the set of trapped bicharacteristics at energy $E$, and suppose that $a\in C_0^\infty(T^*X)$ is identically $1$ near $K_E$. Then there exist $C_1, h_1>0$ such that for any $\eps > 0$, $h \in (0,h_1]$ we have the following nontrapping estimate:
\begin{equation}\label{e:firstconc}
\|\la r \ra^{-1/2 - \delta}(1 - \Op(a))(h^2\Delta_g + V - E - i\eps)^{-1}(1 - \Op(a))\la r \ra^{-1/2 - \delta}\|_{L^2(X) \to L^2(X)} \le C_1 h^{-1}.
\end{equation}
\end{thm}
Here by bicharacteristics at energy $E$ we mean integral curves in $p^{-1}(E)$ of the Hamiltonian vector field $H_p$ of the Hamiltonian $p = |\xi|^2 + V(x)$, and the trapped ones are those which remain in a compact set for all time. We use the notation $r = r(z) = d_g(z,z_0)$, where $d_g$ is the distance function on $X$ induced by $g$ and $z_0 \in X$ is fixed but arbitrary.

Such results were first obtained by Burq \cite{Burq:Lower}, and were later refined by Cardoso and Vodev \cite{Cardoso-Vodev:Uniform}. The improvement here is that to obtain the nontrapping bound the only condition on that cutoffs is that they vanish microlocally near $K_E$ (while in those papers the cutoffs are functions on the base manifold, and are required to vanish on a large compact set whose size is not effectively controlled), but the assumption \eqref{e:polybd} is not needed in \cite{Burq:Lower, Cardoso-Vodev:Uniform}. 

The assumption \eqref{e:polybd} is not true in general. Indeed, when there is elliptic (stable) trapping we have instead $\limsup_{h \to 0}\|\chi_0(h^2\Delta_g + V - E - i\eps)^{-1}\chi_0\|_{L^2(X) \to L^2(X)} \ge e^{1/(Ch)}$ (this has been well known for a long time -- see e.g. \cite{ralston} for an example and \cite{Bony-Burq-Ramond} for a recent introduction to the subject of semiclassical resolvent estimates). Nonetheless, \eqref{e:polybd} is satisfied for many hyperbolic trapped geometries, including those studied in \cite{Nonnenmacher-Zworski:Quantum,  wz}. See \cite[Theorem 6.1]{DaVa} for \eqref{e:polybd} in the asymptotically hyperbolic case, and see \cite{d} and \cite[Corollary 1]{wz} for the asymptotically conic case. Bony and Petkov \cite{Bony-Petkov:Resolvent} prove \eqref{e:polybd} for a general ``black box'' perturbation of the Laplacian in $\R^n$ assuming only that there is a resonance-free strip, and it is likely that this condition suffices for asymptotically conic or hyperbolic manifolds as well. It is an open problem to find the optimal general bound implied by a resonance free strip, or to find assumptions under which one has a polynomial bound \eqref{e:polybd} but no resonance free strip.

We remark that, in the setting of \cite{Nonnenmacher-Zworski:Quantum,  wz},  \eqref{e:polybd} holds with $C_0h^{-k}$ replaced by $C_0(\log h^{-1})h^{-1}$, and so the improvement in our result is only of a factor of $\log (1/h)$.  On the other hand, in \cite{Bony-Burq-Ramond}, Bony, Burq and Ramond prove that for $P$ a semiclassical Schr\"odinger operator on $\R^n$, the presence of a single trapped trajectory implies that
\[\sup_{\lambda \in [-\eps,\eps]} \|\chi(P-\lambda)^{-1}\chi\| \ge \frac{\log(1/h)}{Ch},\]
provided $\chi \in C_0^\infty(X)$ is $1$ on the projection of the trapped set, so in this case (and probably in general) the improvement in Theorem \ref{t:cor} is of no less than a factor of $\log(1/h)$. In \cite{cw}, Christianson and Wunsch give some examples of surfaces of revolution on which a resolvent estimate holds with a bound $h^{-k}$ (but not $C_0(\log h^{-1})h^{-1}$).

We actually prove our main theorem in the following still more general setting. Suppose $X$ is a manifold, $P \in \Psi^{m,0}(X)$ a self adjoint, order $m>0$, semiclassical pseudodifferential operator on $X$, with principal symbol $p$. For $I \subset \R$ compact and fixed, denote the characteristic set by $\Sigma=p^{-1}(I)$, and suppose that the projection to the base, $\pi\colon\Sigma \to X$, is proper (it is sufficient, for example, to have $p$ classically elliptic). Suppose that $\Gamma\Subset T^*X$ is invariant under the bicharacteristic flow in $\Sigma$. Define the {\em forward, resp.\ backward flowout} $\Gamma_+$, resp.\ $\Gamma_-$, of $\Gamma$ as the set of points $\rho$ in the characteristic set, $\Sigma$, from which the backward, resp.\ forward bicharacteristic segments tend to $\Gamma$, i.e.\ for any neighborhood $O$ of $\Gamma$ there exists $T>0$ such that $-t\geq T$, resp.\ $t\geq T$, implies $\gamma(t)\in O$, where $\gamma$ is the bicharacteristic with $\gamma(0)=\rho$. Here we think of $\Gamma$ as the trapped set or as part of the trapped set, hence points in $\Gamma_-$, resp.\ $\Gamma_+$ are backward, resp.\ forward, trapped, explaining the notation. Suppose $V$, $W$ are neighborhoods of $\Gamma$ with $\overline{V}\subset W$, $\overline{W}$ compact. Suppose also that
\begin{equation}\label{e:dynamicassumption}\begin{split}&\textrm{If $\rho\in W \setminus \Gamma_+$, resp. $\rho\in W \setminus \Gamma_-$,}\\ \textrm{then the backward, }& \textrm{resp. forward bicharacteristic from $\rho$ intersects $W\setminus \overline{V}$.}\end{split}
\end{equation}
The main result of the paper, from which the other results follow, is the following:

\begin{thm}\label{t:main}
Suppose that $u$ is a polynomially bounded family (in $h$) of distributions with $(P-\lambda)u=f$, $\re \lambda \in I$ and $\im \lambda \ge -\Oh(h^\infty)$. Suppose $f$ is $\Oh(1)$ in $L^2$ microlocally on $W$, $\WFh(f)\cap \overline{V}=\emptyset$, and $u$ is in $\Oh(h^{-1})$ microlocally on $W \cap \Gamma_- \setminus \overline{V}$, then $u$ is $\Oh(h^{-1})$ microlocally on $W\cap \Gamma_+ \setminus \Gamma$.
\end{thm}

Note that there is no conclusion on $u$ at $\Gamma$; typically it will be merely polynomially bounded. However, to obtain $\Oh(h^{-1})$ bounds for $u$ on $\Gamma_+$ we only needed to assume $\Oh(h^{-1})$ bounds for $u$ on $\Gamma_-$ and nowhere else.  Note also that by the propagation of singularities, if $u$ is $\Oh(h^{-1})$ at one point on any bicharacteristic, then it is such on the whole forward bicharacteristic. If $|\im \lambda| = \Oh(h^\infty)$ then the same is true for backward bicharacteristics.

In certain more complicated geometries it is possible to apply Theorem~\ref{t:main}  with $\Gamma$ a proper subset of $\tilde\Gamma$ which is not a connected component, allowing both $\supp a$ and $\supp b$ to intersect $\tilde \Gamma$. More specifically, when applying Theorem~\ref{t:main}, $W \cap \tilde \Gamma$ does not have to be a subset of $\Gamma$. This is because of the possibility of interesting dynamics within $\tilde \Gamma$, for example a trajectory which tends to different closed orbits as $t \to \pm \infty$, and thus is trapped. In this case $\Gamma$ could be one of the closed orbits. In \S\ref{s:contrived} we give an (admittedly contrived) example of this.

An interesting open question concerns the optimality of the condition $\im \lambda \ge - \Oh(h^\infty)$ in Theorem \ref{t:cor}. That some such condition is needed is suggested by the following result of Petkov and Stoyanov \cite[\S 4]{Petkov-Stoyanov:Singularities} for obstacle scattering on $\R^n$ with $n$ odd. They show that if the cutoff resolvent continues analytically to $\{|\re \lambda| \le E, \im \lambda \ge - C h \log(1/h)\}$, then a polynomial bound for $\|\chi(h^2\Delta_g - \lambda)^{-1}\chi\|$ in this range of $\lambda$, even for $\chi \in C_0^\infty(X)$ supported very far from the trapped set, implies the same bound for a general $\chi \in C_0^\infty(X)$, with possibly worse constant $C$. In other words, no improvement is possible for such a large range of $\lambda$. In fact, we have been informed by Vesselin Petkov that the assumption that the cutoff resolvent continues analytically to a logarithmic region can be replaced by the same assumption on a strip, using the same method.

The general idea of proving propagation estimates through trapped sets via commutator estimates is that near the trapped set $\Gamma$, where we cannot expect any improvement over a priori bounds, the commutator should vanish, which is in particular the case if the commutant is microlocally near $\Gamma$ a (possibly $h$-dependent) multiple of the identity operator. Such a commutant, which is in addition decreasing along the Hamilton flow elsewhere on the characteristic set, at least apart from backward non-trapped bicharacteristics (where one has $\Oh(h^{-1})$ a priori bounds), can indeed be constructed, see \S\ref{s:general}. In fact, under additional geometric assumptions, namely a certain convexity (which also plays a role in \cite{Burq:Lower, Cardoso-Vodev:Uniform}), one can use as commutants cutoff functions which are constant on the projection of the trapped set to the base manifold $X$; this is the special case we consider in \S\ref{s:special}.

This scheme has much in common with an aspect of $N$-particle scattering. In order to prove asymptotic completeness for the short range $N$-particle problem, it suffices to obtain improved weighted estimates in $\langle z\rangle^{1/2}L^2$, where $z$ is the variable on $\RR^{Nd}$ (or $\RR^{(N-1)d}$),  away from the {\em radial set} of the Hamilton vector field of the various subsystems, also called the {\em propagation set} of Sigal and Soffer \cite{Sigal-Soffer:N} (the corresponding global weighted estimate is in $\langle z\rangle^{1/2+\eps}L^2$, and the improvement though small is crucial in the argument). Since there cannot be an improvement at the radial set, the commutant used in the proof must commute microlocally with the Hamiltonian there. Similarly, in our case, there cannot be an improvement at the trapped set, and so our commutant must commute microlocally with $P$ there. In the $N$-particle setting,
the weights $\langle z\rangle^s$ do not commute with the Hamiltonians, unlike
the weights $h^{-s}$ in the semiclassical setting, so, to obtain a microlocally
commuting commutant, one needs to work with $s=0$, which in turn
gives rise to weighted estimates {\em only in the particular weighted space}
$\langle z\rangle^{1/2}L^2$ microlocally away from the radial set.
See \cite{Sigal-Soffer:N} and \cite{Derezinski-Gerard} for a discussion of asymptotic completeness, and \cite{Vasy:Geometry} for a discussion of the proof of this estimate from a microlocal point of view.

More standard escape function methods can prove related but weaker results. For example in \cite[Lemma 2.2]{bgh}, Burq, Guillarmou and Hassell use a positive commutator argument with a global escape function (see also \cite[Appendix]{GeSj} for a more general version of the same escape function) to prove local smoothing away from a trapped set. This corresponds in our setting to a resolvent estimate for $\im \lambda \ge Ch$ (i.e. not too close to the spectrum), and in this range of $\lambda$ one has more flexibility in the behavior of the escape function near infinity, because the resolvent has good mapping properties for a wider range of pairs of weighted spaces. This difference is most significant in the case of an asymptotically hyperbolic space, such as the hyperbolic cylinder of the example at the beginning of the introduction, because here it does not seem to be possible to modify the global escape function so as to give uniform estimates up to the spectrum. In Theorem~\ref{t:main} the global construction is replaced by the assumption that $u$ is $\Oh(h^{-1})$ on $\Gamma_-$ away from $\Gamma$. In the setting of resolvent estimates, this can be proved by commutator estimates on an asymptotically conic space (see \cite{vz}, \cite{d}), but on more general spaces other methods may be more convenient, or even necessary. For instance, in \cite{Melrose-SaBarreto-Vasy:Semiclassical}, Melrose, S\'a Barreto and the second author construct a parametrix for manifolds which are strongly asymptotically hyperbolic in a certain sense (see \S\ref{s:inf}), and the Lagrangian structure of this parametrix implies the semiclassically outgoing property. In \cite{v1,v2}, the second author proves the same result on more general even asymptotically hyperbolic spaces (in the sense of \S\ref{s:defnot}) using commutator methods, but in order to do this he considers a conjugated operator on a modified space. 

The other advantage over global escape function methods is that, because our assumptions and constructions are completely microlocalized to a neighborhood of $\Gamma$ (which may be a proper subset of the full trapped set), our method can give more precise information about a solution $u$ to $Pu=f$ in the case where different estimates on $f$ are available on different parts of $T^*X$. The key point is that in the Theorem \ref{t:cor} and in the example at the beginning of the introduction we apply Theorem~\ref{t:main} with $\Gamma$ a proper subset of the trapped set. 

The structure of this paper is the following. In \S\ref{s:defnot} we give definitions and notation. In \S\ref{s:special}, we prove a special case of Theorem~\ref{t:first} in which the ideas of the proof are more transparent. In \S\ref{s:general} we prove Theorem~\ref{t:main}. In \S\ref{s:app} we prove Theorem \ref{t:cor} and give an example in which Theorem~\ref{t:main} can be applied to a subset of the trapped set which is not a connected component. In \S\ref{s:outgoing} we discuss the semiclassically outgoing assumption and give examples of situations where it is satisfied, and we deduce Theorem \ref{t:first} from Theorem~\ref{t:cor}.

We are grateful to Maciej Zworski for his interest in this project and for several stimulating discussions about polynomially bounded resolvents, and also to Vesselin Petkov for several interesting discussions about related results and problems in obstacle scattering. Thanks also to the anonymous referee for the suggestion to include a discussion of noncompactly supported weights.

\section{Definitions and notation}\label{s:defnot}

\begin{itemize}
\item Let $X$ be the interior of $\overline{X}$, a compact manifold with boundary and let $x$ be a boundary defining function on $\overline{X}$, that is a function $x \in C^\infty(\overline{X};[0,\infty))$ with $x^{-1}(0) = \D \overline{X}$ and $dx|_{\D \overline{X}} \ne 0$. Let $g$ be a Riemannian metric on $X$. We say that $(X,g)$ is \textit{asymptotically conic} (in the sense of the large end of a cone) if we have a product decomposition of $\overline X$ near $\D \overline X$ of the form $[0,\eps)_x \times \D X$ where the metric $g$ takes the form
\[g = \frac {dx^2}{x^4} + \frac{\tilde g}{x^2},\]
where $\tilde g$ is a symmetric cotensor smooth up to $\D X$ with $\tilde g|_{\D X}$ a metric. Such metrics are also sometimes called scattering metrics.

If on the other hand
\[g = \frac {dx^2}{x^2} + \frac{\tilde g}{x^2},\]
where $\tilde g$ is a symmetric cotensor smooth up to $\D X$ with $\tilde g|_{\D X}$ a metric, and with $\tilde g$ even in $x$, we say $(X,g)$ is \textit{asymptotically hyperbolic}. See \cite[Definition 1.2]{g} for a more invariant way to phrase this definition.

\item We denote by $\pi$ the projection $T^*X \to X$.

\item If $u$ is a function, $\|u\|$ denotes the $L^2(X)$ norm. If $A$ is an operator, $\|A\|$ denotes the  $L^2(X) \to L^2(X)$ norm. Angle brackets $\la \cdot,\cdot \ra$ denote the inner product on $L^2(X)$.

\item We say that a family of functions $u=(u_h)_{h\in (0,1)}$ on $X$ is \textit{polynomially bounded} if $\|u\| \le C h^{-N}$ for some $N$.

\item By $u \in \Oh(h^\infty)$ or $u = \Oh(h^\infty)$ we mean $\|u\| \le C_N h^N$ for every $N$ and for $h \in (0,1)$. By $\mu \ge - \Oh(h^\infty)$ we mean $\mu \ge -C_Nh^N$ for every $N$, $h \in (0,1)$.

\item  For $a =(a_h)_{h\in(0,1)} \in C^\infty(T^*X)$, we say $a \in S^{m,k}(X)$ if $a$ obeys 
\[\left|\D_z^\alpha \D_\zeta^\beta a \right| \le C_{\alpha,\beta} h^{-k}(1+|\zeta|^2)^{(m-|\beta|)/2},\]
in any coordinate patch, where the $z$ are coordinates in the base and $\zeta$ are coordinates in the fiber, and $\alpha,\beta$ are multiindices. Acting on $u\in\CI_0(X)$ compactly supported in a patch, $\Op(a)$ is a \textit{semiclassical quantization} given in local coordinates by
\[\Op(a)u(z) = \frac 1 {(2\pi h)^n} \int e^{iz\zeta/h}a(z,\zeta)\widehat {u}(\zeta)d\zeta.\]
The operator $\Op(a)$ can be extended to general $u \in C_0^\infty(X)$ by using a partition of unity subordinate to an atlas of charts, and we say $\Op(a) \in \Psi^{m,k}(X)$. The quantization depends on the choice of atlas and on the partition of unity, but the classes $S^{m,k}$ and $\Psi^{m,k}$ do not. Moreover, for given $A = \Op(a) \in \Psi^{m,k}$, the \textit{principal symbol}, defined to be the equivalence class of $a$ in $S^{m,k} \slash S^{m-1,k-1}$, is also invariantly defined. If $A \in \Psi^{m,k}$ and $B \in \Psi^{m',k'}$, then $[A,B] \in \Psi^{m+m'-1,k+k'-1}$ and has principal symbol $\frac h i H_ab$. See, for example, \cite{ds,ez} for more information on these and other results from semiclassical analysis discussed in this section.

\item By \textit{bicharacteristic} we always mean a bicharacteristic of $P$, that is an integral curve of the Hamiltonian vector field of $p$ (the principal symbol of $P$), contained in $p^{-1}(I)$. We denote by $\gamma_\rho$ the bicharacteristic at $\rho$ (or from $\rho$ or through $\rho$), which is defined by the properties  $\gamma_\rho'(t) = H_p (\gamma_\rho(t))$ and $\gamma_\rho(0) = \rho$. We denote this by $\gamma_\rho^\pm$ the restriction of $\gamma_\rho$ to $\{\pm t\ge0\}$.  We call $\gamma^+_\rho$ the \textit{forward bicharacteristic} and $\gamma^-_\rho$ the \textit{backward bicharacteristic}.

\item For $\Gamma\Subset T^*X$ invariant under the bicharacteristic flow, we define the {\em forward, resp.\ backward flowout} $\Gamma_+$, resp.\ $\Gamma_-$, of $\Gamma$ as the set of points $\rho \in T^*X$ from which the backward, resp.\ forward bicharacteristic segments tend to $\Gamma$, i.e.\ for any neighborhood $O$ of $\Gamma$ there exists $T>0$ such that $-t\geq T$, resp.\ $t\geq T$, implies $\gamma_\rho(t)\in O$. Here we think of $\Gamma$ as the trapped set or as part of the trapped, hence points in $\Gamma_-$, resp.\ $\Gamma_+$ are backward, resp.\ forward, trapped, explaining the notation.

\item For $E \subset T^*X$, we denote by $\Gamma_\pm^E$ the set $\{\rho \in \Gamma_\pm:\gamma_\rho^\mp \in E\}$. Note that $E \subset F \Rightarrow \Gamma_\pm^E \subset \Gamma_\pm^F$, and that $\Gamma_\pm \setminus \Gamma_\pm^U$ is closed when $U$ is open.

\item For $k \in \R \cup \{\infty\}$, we say that $u$ is $\Oh(h^k)$ at a point $\rho \in T^*X$ if there exists $a \in C_0^\infty(T^*X)$ with $a(\rho) \ne 0$ such that $\|\Op(a)u\| = \Oh(h^k)$. We say that $u$ is $\Oh(h^k)$ on a set $E \subset T^*X$ if it is $\Oh(h^k)$ at each point in $E$. Observe that if $E$ is compact we may sum finitely many such functions $|a|^2$ to obtain $b \in C_0^\infty(T^*X)$ which is nonvanishing on $E$ such that $\|\Op(b)u\| = \Oh(h^k)$. Observe also that the set on which $u$ is $\Oh(h^k)$ is open for any $k$.

\item The \textit{semiclassical wave front set}, $\WFh(u)$, is defined for polynomially bounded
$u$ as follows: a point $\rho\in T^*X$ is not in $\WFh(u)$ if $u$ is $\Oh(h^\infty)$ at $\rho$. One can also extend the definition to $\rho\in S^*X$ (thought of as the cosphere bundle at fiber-infinity in $T^*X$); then $\WFh(u)=\emptyset$ implies $u=\Oh(h^\infty)$ (in $L^2$). 

\item The \textit{microsupport}, $\WF_h'A$, is defined for $A = \Op(a) \in \Psi^{m,k}(X)$ as follows: a point $\rho \in T^*X$ is not in $\WF_h'A$ if $|\D^\alpha a| = \Oh(h^\infty)$ near $\rho$ for any multiindex $\alpha$. For any $B \in \Psi^{m',k'}$, we have $\WF'_h([A,B]) \subset \supp da$, and for any $u$ polynomially bounded we have $\WF_h Au \subset \WF_h'A \cap \WF_h u$.

\item If $A \in \Psi^{m,k}$ has principal symbol $a = a_h$, we say that $A$ (or $a$) is \textit{elliptic} at a point $\rho \in T^*X$ if $|a(\rho')| \ge C h^{-k}$ for $\rho'$ near $\rho$ and $h>0$ sufficiently small. We say $A$ (or $a$) is elliptic on a set $E \Subset T^*X$ if it is elliptic at each point in $E$, and we automatically get a uniform estimate $|a(\rho')| \ge C h^{-k}$ for $\rho' \in E$. \textit{Microlocal elliptic regularity} states that if $Au = f$ with $u$  polynomically bounded, then if $f$ is $\Oh(1)$ on a set $E$ and if $A$ is elliptic on $E$, then $u$ is $\Oh(1)$ on $E$.

\item Let $P \in \Psi^{m,0}(X)$ be a self adjoint, order $m>0$, semiclassical pseudodifferential operator on $X$, with principal symbol $p$. For $I \subset \R$ compact and fixed, denote the characteristic set by $\Sigma=p^{-1}(I)$, and suppose that the projection to the base, $\pi\colon\Sigma \to X$, is proper (it is sufficient, for example, to have $p$ classically elliptic). For $w \in C^\infty(T^*X;[0,\infty))$. We say that a point $\rho \in \Sigma$ is \emph{backward nontrapped} with respect to $p-iw$, if either $w(\gamma_\rho(t))>0$ for some $t < 0$ or if for any $K \Subset T^*X$, there exists $T_K<0$ such that $\gamma_\rho(t) \not\in K$ whenever $t \le T_K$.

\item We say that a polynomially bounded resolvent family $R_h(\lambda)$ is \textit{semiclassically outgoing} with loss of $h^{-1}$ at backward nontrapped points if the following holds. If $u = R_h(\lambda)f$ with $f$ compactly supported and $\rho$ lies on a backward nontrapped bicharacteristic, and if $f$ is $\Oh(1)$ on the backward flowout of $\rho$, then $u$ is $\Oh(h^{-1})$ at $\rho$. \emph{In the rest of the paper we will often write simply `semiclassically outgoing' for brevity, but note that this condition is stronger than the one in \cite{DaVa} because the loss is specified to be $h^{-1}$.} This condition is discussed in \S\ref{s:outgoing}.

\item In this setting \textit{propagation of singularities} states that if $u = R_h(\lambda)f$, and $u$ is $\Oh(h^k)$ at $\rho$ and $f$ is $\Oh(h^{k+1})$ on $\gamma_\rho([0,T])$ for some $T > 0$, then $u$ is $\Oh(h^k)$ at $\gamma_\rho(T)$.
\end{itemize}

\section{A microlocal proof in a non-microlocal setting}\label{s:special}

In the next section we prove our general result. In this section we prove a special case of Theorem~\ref{t:first}, indeed essentially a special case of \cite[(2.28)]{Burq:Lower} and \cite[(1.5)]{Cardoso-Vodev:Uniform}, in which the ideas are more transparent. We assume the resolvent is polynomially bounded and semiclassically outgoing at backward nontrapped points. However, we do not assume a specific structure at infinity: this is replaced by the semiclassically outgoing assumption, which is currently known for certain asymptotically conic and hyperbolic infinities (see \S\ref{s:outgoing}), but should hold in other cases as well. In this section we make a convexity assumption in an annular neighborhood of the trapped set, but this assumption is removed in the next section.

Let $X$ be a manifold without boundary, $g$ a complete metric on $X$, and $P$ a self-adjoint semiclassical Schr\"odinger operator on $X$. Assume that there exists a small family of convex compact hypersurfaces which enclose the trapped set in the following sense. Fix $I \subset \R$ compact and $x \in C^\infty(X)$ such that $\{x\ge1\}$ is compact and such that the trapped set $\Gamma$ (i.e. the set of precompact bicharacteristics in $p^{-1}(I)$) sits inside $\{x>5\}$. Suppose that the bicharacteristics $\gamma$ of $P$  in $p^{-1}(I)$ satisfy the convexity
assumption
\begin{equation}\label{eq:convexity}
1< x(\gamma(t)) < 5, \, \dot x(\gamma(t))=0\Rightarrow \ddot x(\gamma(t))<0.
\end{equation}
Here we note that if $f$ is a $\CI$
function on $[0,\infty)$ with $f'>0$,
and $x$ satisfies \eqref{eq:convexity} then so does
$f\circ x$.
In particular the specific constants above and below (such as $x<5$)
are chosen only for convenience,
and can be replaced by arbitrary constants that preserve the ordering. In examples $x$ might be the reciprocal of a function which measures distance to a given point, or more generally $x$ might be a boundary defining function.

\begin{figure}[htb]
\includegraphics[width=15cm]{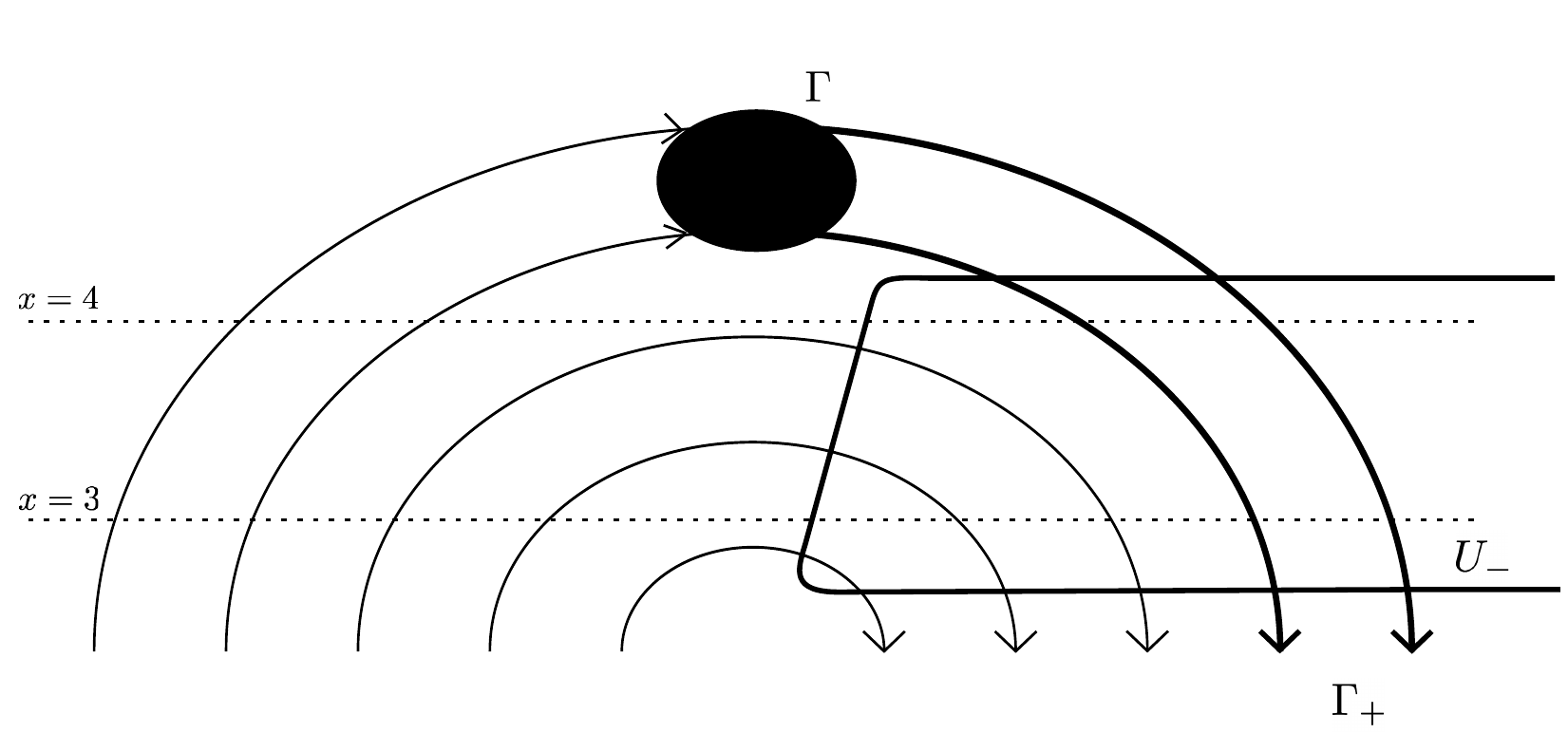}
\caption{The characteristic set $\Sigma = p^{-1}(I)$ in \S\ref{s:special}, with bicharacteristics shown as arrows. The first paragraph of the proof of Proposition~\ref{p:specialcase} reduces the problem to showing that $u$ is $\Oh(h^{-1})$ on a compact subset of $\Gamma_+ \cap \{3 < x < 4\}$. To do this we use as a commutator a cutoff function $\chi = \chi(x)$ which is $1$ for $x \ge 4$, $0$ for $x \le 3$, and monotonic in between. This commutator has a uniform sign on the part of $U_- \subset \{H_px < -c\}$ where $\chi'$ is bounded away from $0$. If $\supp \chi'$ is sufficiently large and $c >0$ is sufficiently small this set contains the compact subset of $\Gamma_+ \cap \{3 < x < 4\}$ in which we are interested.}
\end{figure}

\begin{prop}\label{p:specialcase} Let $(X,g)$, $P$, $I$,  and $x$ be as in the above paragraph. Assume that there exists $N >0$, $\chi_0 \in C_0^\infty(X)$ with $\chi_0 =1$ on $\{x\ge1\}$, and $C > 0$ such that the resolvent satisfies 
\[\|\chi_0 R_h(\lambda)\chi_0\| \le C h^{-N},\]
for $\lambda \in D \subset \{\re \lambda \in I$, $\im \lambda \ge -\Oh(h^\infty)\}$.
Assume that the resolvent is semiclassically outgoing at backward nontrapped points. Then if $\chi_1 \in C_0^\infty(X)$ is supported in $\{3<x<4\}$ we have
\[\|\chi_1 R_h(\lambda)\chi_1\| \le C h^{-1},\]
for $\lambda \in D$.
\end{prop}

\begin{proof}
We show first that if $v$ is compactly supported with $\|v\| = \Oh(1)$, and
$u=R_h(\lambda)f$, where $f=\chi_1 v$, then $u$ is $\Oh(h^{-1})$ on $T^*\supp\chi_1$. By our hypothesis,
$u$ is polynomially bounded in $h$, namely is $\Oh(h^{-N})$,
$(P_h-\lambda)u_h=f_h$ is $\Oh(1)$ and compactly supported, and
$(P_h-\lambda)u_h=0$ in $x>2$. Thus, by microlocal elliptic regularity,
using the polynomial bound, $u_h$ is $\Oh(1)$ away from the characteristic set,
$\Sigma=\{p-\re\lambda \in I\}$.
Moreover, by the semiclassical outgoing assumption, if $\rho$ is a point
in $\Sigma$ but not in $\Gamma_+$, then $(P-\lambda)u=f$ being
$\Oh(1)$ implies that $u$ is $\Oh(h^{-1})$ along $\gamma_\rho$, i.e. a
non-trapping estimate holds microlocally along $\gamma_\rho$. It remains to show that $u$ is $\Oh(h^{-1})$ on points in $\Sigma \cap \Gamma_+ \cap \{3<x<4\}$.

To do so, we proceed inductively, assuming that for some $k\leq -3/2$, $u$ is $\Oh(h^k)$ in a compact subset of $\{3<x<4\}$, and show that it is in fact $\Oh(h^{k+1/2})$ on a slightly smaller subset. Note that the last assumption automatically holds with $k\leq -N$ by the a priori
polynomial bound assumption, and thus the proof of the proposition is complete once
the inductive step is shown.

Take $\chi=\chi(x)\geq 0$ to be a 
function such that $\chi\equiv 1$ in $x\geq 4$, $\chi\equiv 0$ in $x\leq 
3$, and $\chi$ is a increasing function of $x$, and $\chi'=\psi^2$ with
$\psi$ smooth. By microlocal elliptic regularity,
$\WFh(u)\cap\supp\chi$ is a subset of the characteristic set of $P_h-\lambda$.
Then consider
\begin{equation}\label{eq:commutator-id}
\langle \chi u,(P-\lambda)u\rangle-\langle 
\chi (P-\lambda)u,u\rangle
=\langle [P,\chi] u,u\rangle+\langle 2i \im\lambda \chi u,u\rangle.
\end{equation}
The left hand side {\em vanishes} in view of the 
support properties of $\chi$ and $f=(P-\lambda)u$. Since $\im\lambda\geq
-\Oh(h^\infty)$, i.e.\ for all $M$ there is $C$ such that
$\im\lambda\geq -Ch^M$,
we thus conclude that
$$
\langle i[P,\chi] u,u\rangle\geq -\Oh(h^\infty).
$$
The semiclassical principal 
symbol of 
$[P,\chi]$ is
$$
\frac{1}{i}h H_p\chi=\frac{1}{i}h \chi' H_p x.
$$
Letting $c>0$ to be determined later on,
we now use a partition of unity for $T^*X$ corresponding to an open
cover which in a neighborhood of the characteristic
set  over $\{3\leq x\leq 4\}$ is essentially given in terms of the sign of
$H_p x$. So consider a neighborhood of the characteristic set  over $\{3\leq x\leq 4\}$ with compact
closure $K$, and let $O$ be a neighborhood of $K$ with compact closure, and
consider the open cover of $T^*X$ by
$$
U_-\Def\{\rho\in O:\ H_p x(\rho)<-c\},\ U_+\Def\{\rho\in O:\ H_p x(\rho)>-2c\}\cup (T^*X\setminus K),
$$
and take $\phi_\pm \in C^\infty(T^*X)$ with  $\phi_+^2+\phi_-^2=1$ and $\supp\phi_\pm\subset\supp U_\pm$.
Then $(-H_p x)^{1/2}$ is $\CI$ on $\supp\phi_-$, and
$$
H_p\chi=-((-H_p x)^{1/2}\psi\phi_-)^2+\psi^2 H_p x\phi_+^2,
$$
so with $b\Def(-H_p x)^{1/2}\psi\phi_-$, $e\Def\psi^2 H_p x\phi_+^2$, $B,E\in\Psi^{-\infty,0}(X)$
with principal symbol $b$, resp.\ $e$, and microsupport $\supp b$, resp. $\supp e$.
$$
i[P,\chi]=-hB^*B+hE+h^2F,
$$
where $F\in\Psi^{-\infty,0}(X)$,
so
$$
h\|Bu\|^2=\langle (hE+h^2F-i[P,\chi])u,u\rangle=h\langle Eu,u\rangle+h^2\langle
Fu,u\rangle-2\im\lambda\|\chi^{1/2}u\|^2.
$$
Note that $h^2\langle Fu,u\rangle$ is $\Oh(h^{2+2k})$ by our a priori assumptions.
Thus, if  $u$ is $\Oh(h^{k+1/2})$ on $\WF_h'(E)$ (half an order better than
a priori expected), the same is true for $u$ on the elliptic set of $B$, i.e.\ we have
half an order improvement on the elliptic set of $B$.

So far we worked with arbitrary $c$; however, if $c$ is not suitably chosen,
the assumption on $u$ on $\WF_h'(E)$ is not necessarily satisfied.
Namely, we need to choose
$c$ so that $\WF_h'(E)$ is in the union of the elliptic set with the backward
non-trapped set, where we already have $\Oh(h^{-1})$ bounds on $u$.

To do so we choose $c>0$ sufficiently small so that all bicharacteristics from
points $\rho$ in $\{3\leq x\leq 4\}$ with $(H_p x)(\rho)\geq -2c$
escape to $x<3$ in the 
backward direction without entering the region $x\geq 5$. This is possible
due to convexity and compactness: by convexity, if $H_px(\rho)\geq 0$ implies
that on the backward bicharacteristic through $\rho$, $x$ is decreasing as time decreases, so
by compactness there exists $T>0$ such that if $\rho$ is as above, then at time
$-T$ the bicharacteristics are in $x\leq 2$. Then by compactness again, there
is $c>0$ such that for all $\rho$ with $(H_p x)(\rho)\geq -2c$, at time $-T$
the bicharacteristics are in $x\leq 2.5$. With this choice of $c$,
every point in $\WF_h'(E)$ is backward
non-trapped or elliptic. Thus, for
$k+1/2\leq -1$, one deduces that $u$ is $\Oh(h^{k+1/2})$ on the elliptic
set of $B$. In particular, we conclude that where $\chi'>0$, $u$ is $\Oh(h^{k+1/2})$
since such points are either in the elliptic set of $B$ or of $P-\lambda$, or $(H_p x)(\rho)\geq -2c$
there, and in either case $u$ is $\Oh(h^{-1})$ (here we use $k+1/2\leq -1$).

One can iterate this by shrinking the support of $d\chi$, hence those of $B$ and $E$
and deduce that $u$ is actually $\Oh(h^{-1})$ in any compact subset
of $\{3<x<4\}$ (one has to choose the initial $\chi$ appropriately if this subset
is large). This proves that $u$ is $\Oh(h^{-1})$ on $\supp\chi_1$,
i.e.\ $\|\chi_1 R_h(\lambda)\chi_1 v\|\leq Ch^{-1}$. An application of Banach-Steinhaus
finishes the proof, giving a constant $C$ uniform in $v$.
\end{proof}

We remark that a key point in this argument is that because $(P-\lambda)u=0$ in the 
trapping region, one needs to know nothing about $u$ itself when one 
considers $\langle\chi (P-\lambda)u,u\rangle-\langle \chi u,(P-\lambda)u\rangle$
in \eqref{eq:commutator-id},
at least if $\im\lambda\geq-\Oh(h^\infty)$.
If instead 
$(P-\lambda)u$ is $\Oh(1)$ there, then all one can say is that $u$ is $\Oh(h^{-N})$
which completely destroys the bounds above, i.e.\ gives a loss.

It is worth noting that although we needed $\im\lambda\geq -\Oh(h^\infty)$, in
any region $\im\lambda\geq -Ch^s$, $s>1$, we can do a finite amount of
iteration and improve on the assumption that $u$ is $\Oh(h^{-N})$. However,
it is not clear whether this can give any useful bounds in practice.

\section{The general case}\label{s:general}

In this section we prove Theorem~\ref{t:main}. First observe that if $u$ is $\Oh(h^{-1})$ at a point $\rho \in \Gamma_+$, then it is $\Oh(h^{-1})$ on $\gamma_\rho^+$, the forward bicharacteristic from $\rho$. Hence it suffices to construct a microlocal commutant whose commutator is positive on points $\rho$ such that $\gamma_\rho^-$ is contained in a small neighborhood of $\Gamma$, and merely nonnegative on the rest of $\Gamma_+$. The main constraint on the neighborhood in which we work is that it must be contained in the $U$ of Lemma~\ref{l:nontrap} and Remark~\ref{r:nontrap}. The proof uses an inductive iteration as in \S\ref{s:special}, so in Lemma~\ref{l:escfunc} we introduce open neighborhoods $\Gamma \subset U_1 \Subset U_0 \Subset U$ but no other properties of these neighborhoods will be used, and they may be arbitrarily close to $\Gamma$ and to $\D U$ respectively.

\begin{lem}\label{l:nontrap}
Suppose $U_-$ is a neighborhood of $(\Gamma_-\setminus\Gamma)
\cap(\overline{W}\setminus V)$.
There is a neighborhood $U \subset V$ of $\Gamma$ such that if
$\alpha\in U\setminus\Gamma_+$ then the
backward bicharacteristic from $\alpha$ enters $U_-$.
\end{lem}

\begin{rem}
\label{r:nontrap}Note that from this and from the assumption that $u$ is $\Oh(h^{-1})$ on $\Gamma_-$, it follows that that $u$ is $\Oh(h^{-1})$ on $U \setminus \Gamma_+$, provided $U_-$ is chosen sufficiently small, namely small enough that $u$ is $\Oh(h^{-1})$ on $U_-$. Note also that, because $U \subset V$, we have $\WF_h f \cap U = \emptyset$.
\end{rem}

\begin{proof}
Suppose no such $U$ exists. Then there is a sequence $\alpha_j\in V\setminus\Gamma_+$
such that
$\alpha_j\to \Gamma$ but the backward bicharacteristics $\gamma^-_{\alpha_j}$
through $\alpha_j$
are disjoint from $U_-$; by passing to a subsequence, using the compactness of
$\Gamma$, we may assume that $\alpha_j\to\alpha\in\Gamma$.
By \eqref{e:dynamicassumption}, the bicharacteristics $\gamma_{\alpha_j}^{-}$ enter $W\setminus
\overline{V}\subset \overline{W}\setminus V$, and the latter is compact.
Let $t_j=\sup\{t<0:\ \gamma_{\alpha_j}(t)\in\overline{W}\setminus V\}$, and
let $\beta_j=\gamma_{\alpha_j}(t_j)$, so $\beta_j\in\overline{W}\setminus V$ as the latter set is closed. Moreover, $\beta_j \in \overline V$: indeed $\gamma_{\alpha_j}([t_j,0])$ is connected and contained in $\overline V \cup (T^*X \setminus W)$, a union of disjoint closed sets, and $\gamma_{\alpha_j}(0) \in V \subset \overline{V}$. By the compactness of $\overline{V}$, the
$\beta_j$ have a convergent subsequence, say $\beta_{j_k}$, converging to
some $\beta\in (\overline W \setminus V) \cap \overline{V} = \D V$.

We claim that
$\beta\in\Gamma_-$, which is a contradiction with $\beta_{j_k}\notin U_-$.
Indeed, otherwise, by \eqref{e:dynamicassumption}, the forward bicharacteristic
$\gamma^+_\beta$ from $\beta$ intersects $W\setminus\overline{V}$. Moreover, since $\gamma_\beta(0) = \beta \in \overline{V}$, there is $T>0$ such that $\gamma_\beta(T)
\in W\setminus\overline{V}$. 
Then, for sufficiently large $k$, the same is true
for the forward bicharacteristic at time $T$ from $\beta_{j_k}$ as
$W\setminus\overline{V}$ is open, i.e.\ $\gamma_{\alpha_{j_k}}(t_{j_k}+T)\in W\setminus\overline{V}$.
By the definition of $t_{j_k}$, $t_{j_k}+T>0$, so $t_{j_k}>-T$ for all $k$.
But, if $\gamma_\alpha$ is the bicharacteristic
through $\alpha$, then $\gamma_{\alpha_{j_k}}(t)\to\gamma_\alpha(t)$ uniformly in $[-T,0]$. By passing
to a convergent subsequence of $t_{j_k}$, say $t_{j'_k}$, $\gamma_{\alpha_{j'_k}}(t_{j'_k})\to
\gamma_\alpha(\lim t_{j'_k})\in \Gamma$ by the flow-invariance of $\Gamma$, so
$\beta\in\Gamma$ which contradicts $\beta \not\in V$. Thus, $\beta\in\Gamma_-$, as claimed.
\end{proof}

In the following lemma we construct an escape function $q \in C_0^\infty(T^*X)$ which is constant near $\Gamma$, nonincreasing along $\Gamma_+$, and has $H_pq<0$ on a sufficiently large subset of $\Gamma_+$. This construction is based in part on the construction of a nontrapping escape function in \cite[\S 4]{vz} and on the construction of an escape function away from a trapped set in \cite[Appendix]{GeSj}. We will use a quantization of $q$ as a microlocal commutant in this section, replacing the cutoff function $\chi$ of \S\ref{s:special}.

\begin{lem}\label{l:escfunc}
Let $U_1$and $U_0$ be an open set with $\Gamma \subset U_1 \Subset U_0 \Subset U$. Then there exists a nonnegative function $q \in C_0^\infty(U)$ such that 
\[q = 1 \textrm{ near } \Gamma, \qquad H_p q \le 0 \textrm{ near } \Gamma_+, \qquad H_p q< 0 \textrm{ on } \Gamma_+^{\overline{U_0}} \setminus U_1.\]
Moreover, we can take $q$ such that both $\sqrt q$ and $\sqrt{-H_pq}$ are smooth near $\Gamma_+$.
\end{lem}

Recall that $\Gamma_+^E$ is the set of points $\rho \in \Gamma_+$ whose backward bicharacteristic $\gamma_\rho^-$ is contained in $E$. The condition that $\sqrt q$ and $\sqrt{-H_pq}$ are smooth near $\Gamma_+$ is used only to avoid invoking the sharp G\aa rding inequality.

\begin{figure}[htb]
\includegraphics[width=13cm]{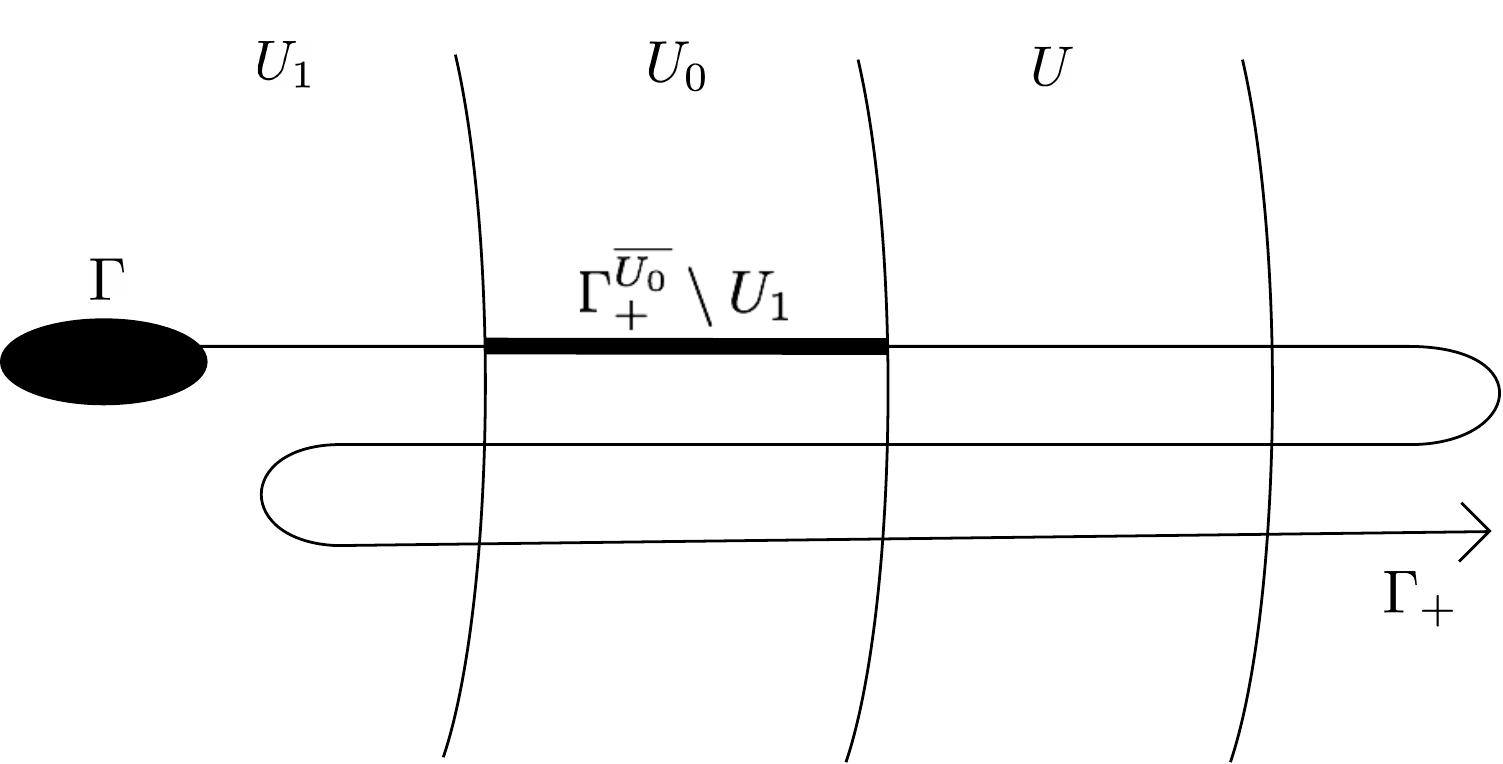}
\caption{We construct $q$ so that it is identically $1$ near $\Gamma$, and then nonincreasing along $\Gamma_+$. We make $q$ strictly decreasing along $\Gamma_+^{\overline{U_0}} \setminus U_1$, and then identically $0$ outside of $U$ (because in this last region Remark~\ref{r:nontrap} provides no information about $u$ so we must not produce any error terms here). Since $q$ must be nonincreasing along $\Gamma_+$ and compactly supported, it must remain $0$ after this point, and in particular  we cannot make $H_p q < 0$ on any of $\Gamma_+ \setminus \Gamma_+^U$.}\label{f:gammazigzag}
\end{figure}

To motivate the statement, we outline how Lemma~\ref{l:escfunc} will be used to prove Theorem~\ref{t:main}. We will see that a positive commutator estimate as in \S\ref{s:special} directly gives us good control of $u$ on $\Gamma_+^{\overline{U_0}} \setminus U_1$, where the commutator is elliptic, up to errors which are of two types. By propagation of singularities we can extend these good estimates to the forward flowout of $\Gamma_+^{\overline{U_0}} \setminus U_1$, namely to $\Gamma_+\setminus U_1$.  The first type of error is in the region away from $\Gamma_+$, where we do not have $H_p q \le 0$, but here we know that $u$ is $\Oh(h^{-1})$ thanks to Remark~\ref{r:nontrap}. The second type of error is in the region where $H_pq \le 0$ but not uniformly bounded away from $0$. We control this error using an iteration as in \S\ref{s:special}.  We will need a finite sequence of $q_j$ (the number of iterations is determined by the polynomial bound on $u$) such that $H_p q_{j+1}< 0$ on $\supp dq_j \cap \Gamma_+$.   To obtain $q_1$ we apply Lemma~\ref{l:escfunc}  with any $U_1, U_0$ satisfying the hypotheses of the lemma. To obtain $q_{j+1}$ from $q_j$ we observe that 
\[\Gamma \subset T^*X\setminus \supp(1-q_j) \subset \supp q_j \subset U,\]
and apply Lemma~\ref{l:escfunc} with a new $U_1,U_0$ such that $U_1 \Subset T^*X \setminus \supp(1-q_j)$ and $\supp q_j \subset U_0$. To simplify notation we will not discuss the iteration in more detail, and will simply use $q$ rather than $q_j$.

\begin{proof}
We will construct a function $\tilde{q}$, smooth in a neighborhood of $\overline{U}$,  satisfying
\[\tilde{q} = 0 \textrm{ near } \Gamma, \quad H_p \tilde{q} \le 0, \quad H_p \tilde{q}|_{ \Gamma_+^{\overline{U_0}}\setminus U_1} < 0, \quad \tilde{q}|_{ \Gamma_+^{\overline{U_0}}\setminus U_1} \ge -1/2, \quad  \tilde{q} \le -2 \textrm{ near } \Gamma_+ \setminus \Gamma_+^{U}.\]
Then we take $f\in C^{\infty}(\R)$ nondecreasing such that $f(t)=t+1$ near $t \ge -1/2$ and $f(t) = 0$ near $t \le -2$. We take further $\chi_q \in C_0^\infty(U;[0,\infty))$ identically $1$ near $\overline{\{\tilde q < 2\}} \cap \Gamma_+$ (note that $\overline{\{\tilde q < 2\}} \cap \Gamma_+ \Subset U$). It then suffices to put
\[q(\rho) \Def \chi_q(\rho) f(\tilde{q}(\rho)).\]
Indeed, that $q$ is nonnegative and identically $1$ near $\Gamma$ is immediate. That $H_p q < 0$ on $\Gamma_+^{\overline{U_0}} \setminus U_1$ follows from the fact that on that set we have $f \circ \tilde{q} = \tilde{q} + 1$ and $\chi = 1$. That $H_p q \le 0$ near $\Gamma_+$ follows from the fact that $\{H_p q>0\} \subset \supp d\chi_q \cap \overline{\{\tilde q < 2\}}$, which is disjoint from $\Gamma_+$.

If $f$ and $\chi_q$ are chosen such that $\sqrt f$ and $\sqrt{\chi_q}$ are smooth, then $\sqrt q$ is smooth. Meanwhile, near $\Gamma_+$, $-H_p q = -(f'\circ \tilde q)H_p \tilde q$, and hence it suffices to make $\sqrt{f'}$ and $\sqrt{-H_p \tilde q}$ smooth. In the case of $f$ it suffices to make $f$ a translation of $e^{-1/t}|_{t >0}$ near the boundary of its support. We will indicate below how to achieve this for $\tilde q$.
 
We take $\tilde{q}$ of the form
\begin{equation}\label{e:q1}\tilde{q} \Def \sum_{k=1}^N q_{\rho_k},\end{equation}
where each $q_{\rho_k}$ is supported near a portion of the bicharacteristic through $\rho_k$, a suitably chosen point in $\Gamma_+^{\overline{U_0}}\setminus U_1$.

To determine the $\rho_k$ we first fix open sets $V_1$ and $V_0 $ with $ \Gamma \subset V_1 \Subset U_1$ and $U_0 \Subset V_0 \Subset U$.  
 We then associate to each  $\rho \in \Gamma^{\overline{U_0}}_+ \setminus U_1$  the following {\em escape times}:
\[T_\rho^{V_1} \Def \inf \{t \in \R: \gamma_\rho(t) \not\in V_1\}, \ T_\rho^{V_0} \Def \inf \{t \in \R: \gamma_\rho(t) \not\in V_0\}, \ T_\rho^U \Def \sup \{t \in \R: \gamma_\rho(t) \in \overline{U}\}.\]
Note that these are finite because of the definition of $\Gamma_+$ and \eqref{e:dynamicassumption}.

Next let $\Sc_\rho$ be a hypersurface through $\rho$ which is transversal to $H_p$ near $\rho$. Then if $U_\rho$ is a sufficiently small neighborhood of $\rho$, the set
\[V_\rho \Def \{\gamma_\alpha(t): \alpha \in U_\rho \cap \Sc_\rho,\, t \in (T_\rho^{V_1} -1,T_\rho^{U} +1)\}\]
is diffeomorphic to $(\Sc_\rho\cap U_\rho) \times (T_\rho^{V_1} -1,T_\rho^{U} +1)$. We use this diffeomorphism to define product coordinates on $V_\rho$. If necessary, shrink $U_\rho$ so that
\[\overline{V_\rho} \cap \{t \le T_\rho^{V_1}\} \cap \overline{U_1} = \emptyset, \qquad \overline{V_\rho} \cap \{t \le T_\rho^{V_0}\} \subset U, \qquad \overline{V_\rho} \cap \{t = T_\rho^{V_0}\} \cap \overline{U_0} = \emptyset.\]
This is possible because $\gamma_\rho(\{t \le T_\rho^{V_1}\}) \cap \overline{U_1} = \emptyset$, $\gamma_\rho(\{t \le T_\rho^{V_0}\}) \subset U$, and $\gamma_\rho(T_\rho^{V_0}) \not\in \overline{U_0}$.

Take $\varphi_\rho \in C_0^\infty(\Sc_\rho \cap U_\rho; [0,1])$ identically $1$ near $\rho$, also considered as a function on $V_\rho$ via the product coordinates, and let $V'_\rho \subset V_\rho$ be an open set  containing $\gamma_\rho([T_\rho^{V_1}-1/2,T_\rho^{U}+1/2])$ such that $\varphi_\rho =1$ on $V_\rho'$. Observe that the $V_\rho'$ with  $\rho \in \Gamma^{\overline{U_0}}_+ \setminus U_1$ are an open cover of $\Gamma_+ \cap \overline{U} \setminus U_1$, because any backward bicharacteristic from a point in $\Gamma_+ \cap \overline{U} \setminus U_1$ enters $\Gamma^{\overline{U_0}}_+ \setminus U_1$ eventually. Now take $\rho_1, \dots, \rho_N$ such that 
\begin{equation}
\label{e:cover} \Gamma_+ \cap \overline{U} \setminus U_1 \subset \bigcup_{k=1}^N V'_{\rho_k}.\end{equation}

For each $\rho \in \{\rho_1, \dots, \rho_N\}$ put
\[q_\rho \Def \chi_\rho\varphi_\rho,\qquad H_pq_\rho = \chi'_\rho\varphi_\rho,\]
where $\chi_\rho \in C^\infty((T_\rho^{V_1} -1,T_\rho^{U} +1))$. 

\begin{figure}[htb]
\includegraphics[width=15cm]{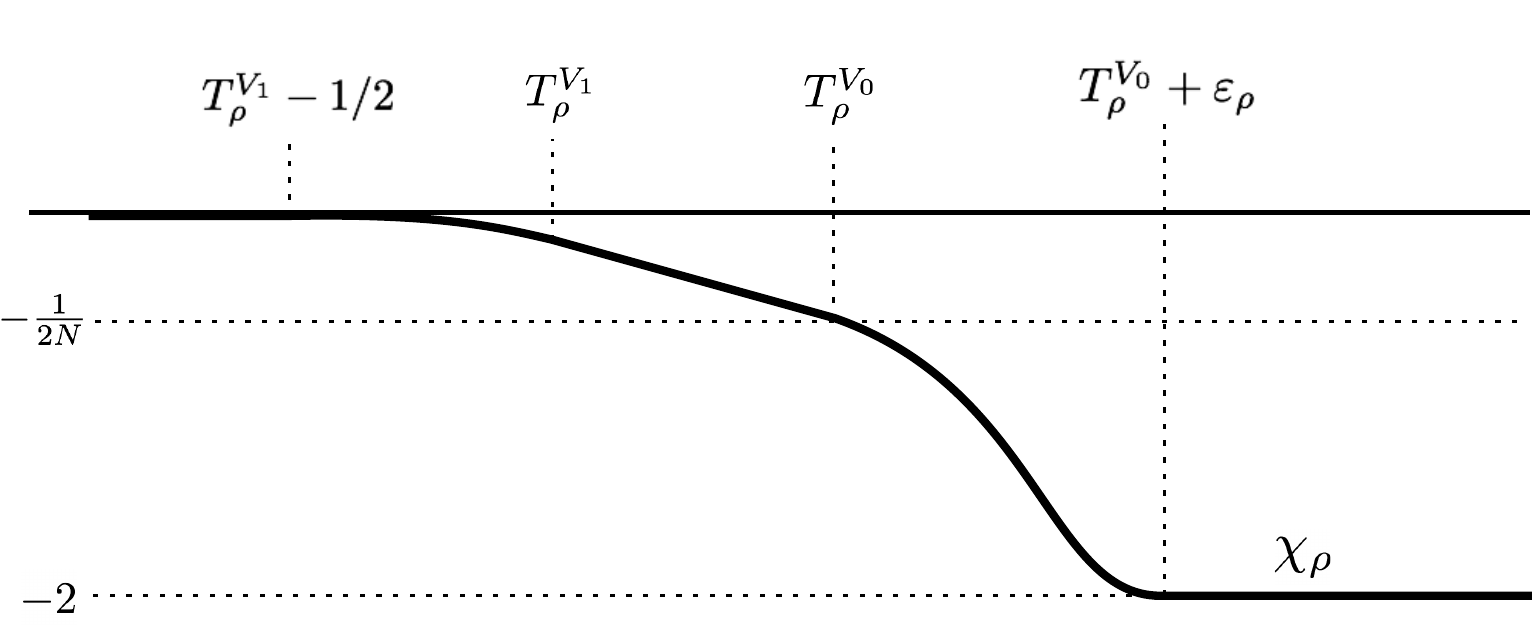}
\caption{The graph of $\chi_\rho$, $\rho \in \{\rho_1,\dots,\rho_N\}$.}\label{f:chirho}
\end{figure}

We further impose that $\chi_\rho$ has $\chi_\rho'(t) \le 0$ for all $t$ and also satisfies 
\begin{enumerate}
\item $\chi_\rho (t) = 0$ for $t \le T_\rho^{V_1} - 1/2$,
\item $\chi_\rho'(t)< 0$, $\chi_\rho(t) \ge -(2N)^{-1}$ for $T_\rho^{V_1} \le t \le T_\rho^{V_0}$,
\item $\chi_\rho(t) = -2$ for $t \ge T_\rho^{V_0} + \eps_\rho$.
\end{enumerate}  
Here $\eps_\rho$ is a positive number less than $1/2$ and small enough that $\gamma_\alpha(t) \in U$ for $\alpha \in U_\rho\cap\Sc_\rho$ and $t \le T_\rho^V + \eps_\rho$.  Such an $\eps_\rho$ exists because $\overline{V_\rho} \cap\{t \le T_\rho^{V_0}\} \subset U$. Note that in condition (2) we use the same $N$ as in \eqref{e:cover}.  Observe that extending $q_\rho$ by  $0$ outside of $V_\rho$ gives a function which is $C^\infty$ near $\overline U$.

We now check that $\tilde{q}$ has the desired properties. That $\tilde{q}=0$ near $\Gamma$ follows from the fact that $\supp \tilde{q} \subset \bigcup V_{\rho_k}$ and each $V_{\rho_k}$ is disjoint from $\Gamma$. That $H_p \tilde{q} \le 0$ follows from $\chi_\rho' \le 0$. That $H_p\tilde{q}<0$ and $\tilde{q} \ge -1/2$ on $\Gamma^{\overline{U_0}}\setminus U_1$ follows from condition (2) on the $\chi_\rho$ and from the covering property \eqref{e:cover}, as well as from the fact that we took care to make $\overline{V_\rho} \cap \{t \ge T_\rho^{V_0}\} \cap \Gamma_+^{\overline{U_0}} = \emptyset$ so none of the summands in \eqref{e:q1} are too negative here. That $\tilde{q} \le -2$ near $\Gamma_+ \setminus \Gamma_+^U$ follows from condition (3) on the $\chi_\rho$ together with \eqref{e:cover}.

To make $\sqrt{-H_p \tilde q}$ smooth, let $\psi(s)=0$ for $s\leq 0$, $\psi(s)=e^{-1/s}$ for $s>0$, and assume as we may that $U_\rho\cap\Sc_\rho$ is a ball with respect to a Euclidean metric (in local coordinates near $\rho$) of radius $r_\rho>0$ around $\rho$. We then choose $\varphi_{\rho}$ to behave like $\psi({r'_\rho}^2-|.|^2)$ with $r'_\rho<r_\rho$ for $|.|$ close to $r'_\rho$, bounded away from $0$ for smaller values of $|.|$, and choose $-\chi'_{\rho}$ to vanish like $\psi$ at the boundary of its support. That  sums of products of such functions have smooth square roots follows from \cite[Lemma 24.4.8]{h}.
\end{proof}

We conclude this section and the proof of Theorem~\ref{t:main} by proving the inductive step in the iteration: if $u$ is $\Oh(h^k)$ on a sufficiently large compact subset of $U \cap\Gamma_+ \setminus \Gamma$, then $u$ is $\Oh(h^{k + 1/2})$ on $\Gamma_+^{\overline U_0} \setminus U_1$, provided $k+ 1/2 \le -1$.

First let $U_-$ be an open neighborhood of $\Gamma_+ \cap \supp q$ which is sufficiently small that $H_pq\le0$ on $U_-$ and that $\sqrt{-H_pq}$ is smooth on $U_-$. Let $U_+$ be an open neighborhood of $\supp q \setminus U_-$ whose closure is disjoint from $\Gamma_+$ and from $T^*X \setminus \overline{U}$. Define $\phi_\pm \in C^\infty(U_+ \cup U_-)$ with $\supp \phi_\pm \subset U_\pm$ and with $\phi_+^2 + \phi_-^2 = 1$ on $U_+ \cup U_-$.

Put
\[b \Def \phi_- \sqrt{-H_pq^2}, \qquad e \Def \phi_+^2 H_pq^2.\]
Let $Q,B,E \in \Psi^{-\infty,0}(X)$ have principal symbols $q,b,e$, and microsupports $\supp q$, $\supp b$, $\supp e$, so that
\[\frac i h [P,Q^*Q] = - B^*B + E + hF,\]
with $F \in \Psi^{-\infty,0}(X)$ such that $\WF_h'F \subset \supp dq \subset U \setminus \Gamma$. But
\[\frac i h \la [P,Q^*Q] u, u \ra = \frac 2 h \im \la Q^*Q(P - \lambda)u, u\ra +\frac 2 h \la Q^*Q \im \lambda u,  u \ra \ge - \Oh(h^\infty) \|u\|^2,\]
where we used $\im \lambda \ge -\Oh(h^\infty)$ and $\supp q \cap \WF_h (P-\lambda)u = \emptyset$. So
\[\|Bu\|^2 \le \la Eu,u \ra + h \la Fu, u \ra + \Oh(h^\infty).\]
But $|\la Eu,u \ra| \le Ch^{-2}$ because $\WF'_hE \cap \Gamma_+ = \emptyset$ allows us to use Remark~\ref{r:nontrap} to conclude that $u$ is $\Oh(h^{-1})$ on $\WF'_hE$. Meanwhile $|\la Fu, u \ra| \le C (h^{-2} + h^{2k})$ because all points of $\WF'_hF$ are either in $U \backslash \Gamma_+$, where we know $u$ is $\Oh(h^{-1})$ from Remark~\ref{r:nontrap}, or on a single compact subset of $U \cap \Gamma_+ \setminus \Gamma$, where we know that $u$ is $\Oh(h^k)$ by inductive hypothesis. Since $b= \sqrt{-H_pq^2}>0$ on $\Gamma_+^{\overline{U_0}}\setminus U_1$, we can use microlocal elliptic regularity to conclude that $u$ is $\Oh(h^{k + 1/2})$ on $\Gamma_+^{\overline U_0} \setminus U_1$, as desired.

\section{Application to resolvent estimates}\label{s:app}

\subsection{Proof of Theorem \ref{t:cor}}\label{s:cor} 

Let $f = \Op(b) v$, $\|v\| = \Oh(1)$, $u = R_h(\lambda)f$.  Let $\Gamma$ be the union of the connected components of the trapped set which intersect the backward bicharacteristic flowout of $\supp a$. Note that if $V$ and $W$ are chosen such that $\overline W$ is disjoint from any other components of the trapped set, then the assumptions of Theorem~\ref{t:main} are satisfied. We must show that for any $\rho \in \supp a$, $u$ is $\Oh(h^{-1})$ at $\rho$. There are three cases.

\begin{enumerate}

\item If $\rho \not\in \Sigma = p^{-1}(I)$, then $u$ is $\Oh(1)$ at $\rho$ by elliptic regularity and the polynomial boundedness of the resolvent.
\item If $\rho$ is backward nontrapped, then $u$ is $\Oh(h^{-1})$ at $\rho$ by the semiclassically outgoing assumption.
\item If $\rho$ is backward trapped, then $\rho \in \Gamma_+$ by the definition of $\Gamma_+$ and by the support property of $a$. Hence $u$ is $\Oh(h^{-1})$ at $\rho$ by Theorem~\ref{t:main}. The assumption in Theorem~\ref{t:main} that $u$ is $\Oh(h^{-1})$ on $\Gamma_-$ follows from case (2) above.

\end{enumerate}

This proves that
\[\|\Op(a)R_h(\lambda)\Op(b) v\| \le Ch^{-1}.\]
The uniformity in $v$ follows from Banach-Steinhaus.

\subsection{Proof of Theorem \ref{t:first}}\label{s:first}

The estimate \cite[(1.5)]{Cardoso-Vodev:Uniform} of Cardoso and Vodev reads, in the notation of Theorem \ref{t:first},
\begin{equation}\label{e:cv}
\|\la r \ra^{-1/2 - \delta}(1 - \chi )(h^2\Delta_g - E - i\eps)^{-1}(1 - \chi )\la r \ra^{-1/2 - \delta}\| \le C h^{-1},
\end{equation}
where $\chi \in C_0^\infty(X)$ is identically $1$ on some (large) compact set. Meanwhile, from Theorem \ref{t:cor} (and using microlocal elliptic regularity) we have
\begin{equation}\label{e:first1}
\| \tilde \chi (1 - \Op(a))(h^2\Delta_g + V - E - i\eps)^{-1}(1 - \Op(a)) \tilde \chi\|\le C h^{-1},
\end{equation}
for any $\tilde \chi \in C_0^\infty(X)$ (see \S\ref{s:inf} for a discussion of the semiclassically outgoing condition in this setting). If we take $\tilde \chi$ to be identically $1$ on a sufficiently large (compact) set, then we can apply the gluing method of \cite{DaVa} to deduce \eqref{e:firstconc} from \eqref{e:cv} and \eqref{e:first1}. Since the proof below follows the proof of \cite[Theorem 2.1]{DaVa} closely we provide only an outline.

After possibly multiplying the boundary defining function $x$ by a large constant, we may assume $\supp V \cup \supp\chi\subset \{x > 4\}$ and that if $\gamma(t)$ is a bicharacteristic of $p_0 = |\xi|_g^2$ in $p_0^{-1}(E)$, then 
\[
\ddot x(\gamma(t)) = 0 \Rightarrow \dot x(\gamma(t)) < 0,
\]
in $\{x > 4\}$. We now take $\tilde \chi$ to be identically $1$ near $\{x \ge 1\}$. Let $\chi_1 \in C^\infty(\R;[0,1])$ be such that $\chi_1 = 1$ near $\{x \ge 3\}$, and $\supp \chi_1 \subset \{x > 2\}$, and let $\chi_0 = 1 - \chi_1$. Define a right parametrix for $P = h^2\Delta_g + V$ by
\[
F \Def \chi_0(x-1) (h^2\Delta_g - \lambda)^{-1}  \chi_0(x) + \chi_1(x+1) (h^2\Delta_g  + V- \lambda)^{-1}\chi_1(x).
\]
We then put
\[\begin{split}
(P - \lambda)F &= \Id + [P,\chi_0(x-1)] (h^2\Delta_g - \lambda)^{-1}  \chi_0(x) + [P,\chi_1(x+1)](h^2\Delta_g  + V- \lambda)^{-1}\chi_1(x) \\&\Def \Id + A_0 + A_1.
\end{split}\]
Now \cite[Lemma 3.1]{DaVa} implies that
\[\|A_0A_1\|_{L^2 \to L^2} = \Oh(h^\infty),\]
so that, using $A_0^2 = A_1^2 = 0$
\[\begin{split}
(P-\lambda)(F - F A_0 - F A_1 + F A_1 A_0) = \Id - A_0 A_1 + A_0 A_1 A_0.
\end{split}\]
Note that the remainder is trivial in the sense that $\|A_0A_1\| + \|A_0A_1A_0\la r \ra^{-1/2 - \delta}\| \le \Oh(h^\infty)$. 
Since \eqref{e:cv} and \eqref{e:first1} imply that 
\[
\|\la r \ra^{-1/2 - \delta}(1 - \Op(a))(F - F A_0 - F A_1 + F A_1 A_0)(1 - \Op(a))\la r \ra^{-1/2 - \delta}\| \le C_1 h^{-1},
\]
this completes the proof.
\subsection{Nontrapping estimates on part of the trapped set}\label{s:contrived}

We now give an  example, although a somewhat unphysical one, in which Theorem~\ref{t:main} can be applied with $\Gamma$ a proper subset of the trapped set but not a connected component. In this example we obtain the nontrapping estimate \eqref{e:ann} for $a$ and $b$ with supports overlapping a certain part of the trapped set. More specifically, we will apply Theorem~\ref{t:main} with $\Gamma$ a union of closed orbits and with part of $\Gamma_+$ or $\Gamma_-$ contained in the trapped set.

Let $y = y(z) \in C^\infty(\R)$ be even, positive-valued, with a nondegenerate local maximum at $0$, and with $y''>0$ outside of a neighborhood of $0$, such that $y''$ changes sign only twice. Let $(X,g)$ be the surface of revolution obtained by revolving the graph of $y$ around the $z$ axis (see Figure~\ref{f:phaseplane}). Suppose this surface is an asymptotically conic or hyperbolic manifold as in \S\ref{s:defnot} (for example, it may be a catenoid outside of a compact set). We will use coordinates $(s,\theta)$ on $X$, where $s=s(z)$ is an arclength parametrization of the graph of $y$ with $s(0)=0$, and $\theta$ measures the angle of revolution. Let $a(s) = y(z)$ and $(\sigma,\mu)$ be dual to $(s,\theta)$. In these coordinates the manifold $(X,g)$ and the geodesic Hamiltonian $p_0$ are given by
\[X = \R_s \times (\R \slash 2\pi \Z)_\theta, \qquad g = ds^2 + a(s)^2 d\theta^2, \qquad p_0 = \sigma^2 + a(s)^{-2} \mu^2.\]
Let $s_0$ be the point in $\{s>0\}$ at which the global minimum of $a$ is attained. The unit speed geodesic flow has six closed orbits along latitude circles: two elliptic orbits at $s=0$ and two hyperbolic orbits at each of $s=\pm s_0$. See Figure~\ref{f:phaseplane} for a sketch of the projection of the bicharacteristic flowlines to the $(s,\sigma)$ plane.
 
\begin{figure}[htb]
\includegraphics{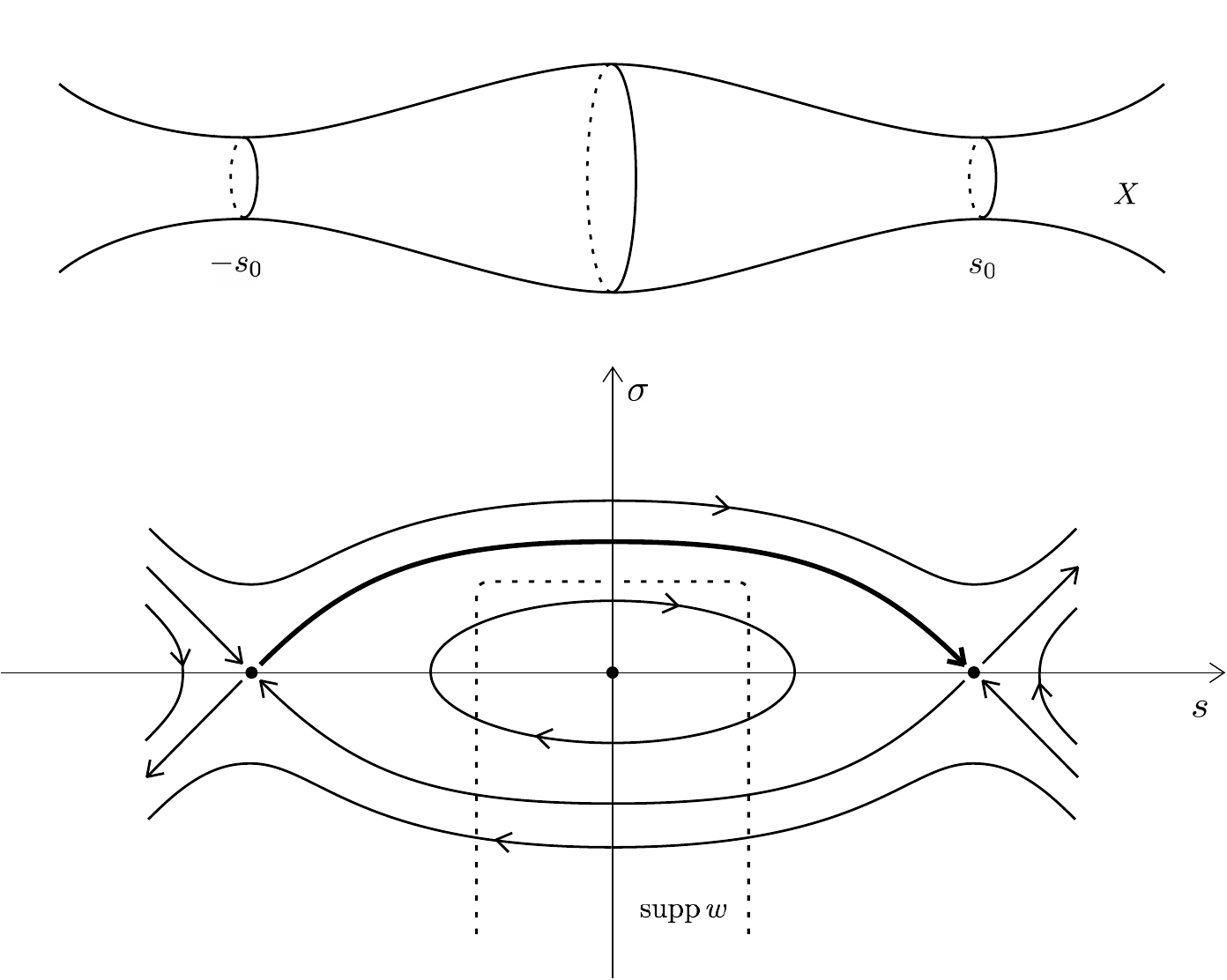}
\caption{The surface of revolution $(X,g)$ with its three geodesic latitude circles, and the unit speed geodesic flow on $S^*X$ projected onto the $(s,\sigma)$ plane. The complex absorbing barrier $w$ is supported inside the dashed outline. In Proposition~\ref{p:contrivedexample} we apply Theorem~\ref{t:main} first with $\Gamma$ taken to be the two hyperbolic closed orbits at $(s,\sigma) = (-s_0,0)$, and then with $\Gamma$ taken to be the orbits at $(s,\sigma) = (s_0,0)$. The darkened arrow is the portion of the trapped set on which we prove a nontrapping resolvent estimate.}\label{f:phaseplane}
\end{figure}

We would like to apply Theorem~\ref{t:main} with $\Gamma$ taken to be one or several of the hyperbolic closed orbits at $s = \pm s_0$. However, the resolvent of the Laplacian on this surface will not be polynomially bounded because of the elliptic trapping, and $\Gamma_-$ in this case will include trapped trajectories on which $\Oh(h^{-1})$ resolvent bounds do not hold, so we introduce a complex absorbing barrier as in \S\ref{s:bar} to suppress some of the trapping. Let $w\in C_0^\infty(T^*X;[0,1])$ be supported as in Figure~\ref{f:phaseplane} and satisfy $w=1$ on $S^*X \cap \{s=0, \, \sigma \le 0\}$. More specifically, we require that $\supp w \subset \{-s_0/2<s<s_0/2\}$, and that $\supp w$ be disjoint from bicharacteristics $\gamma(t)$ with $\lim_{t \to \pm \infty} s(\gamma(t)) = \pm s_0$. Let 
\[P = h^2\Delta_g -1 -iW,\]
where $W \in \Psi^{-\infty,0}(X)$ has principal symbol $w$. In Lemma~\ref{l:polybd} we show that the resolvent of this operator is polynomially bounded. The proof uses microlocal estimates near the hyperbolic orbits originally due to Christianson \cite{c07,c08} together with the gluing method of \cite{DaVa}. To apply the gluing method, we use the following convexity properties of the bicharacteristic flow: If $\gamma(t)$ is a bicharacteristic in $S^*X$, then
\begin{align}\dot s(\gamma(t)) = 0, \,\,\,\,\,\quad\quad \pm s(\gamma(t)) > s_0 \quad &\Rightarrow \quad \pm \ddot s(\gamma(t)) > 0, \label{e:convinf}\\
\dot s(\gamma(t)) = 0, \quad 0 < \pm s(\gamma(t)) < s_0 \quad &\Rightarrow \quad \pm \ddot s(\gamma(t)) <0. \label{e:convcompact}\end{align}

\begin{lem}\label{l:polybd}
For all $\chi_0 \in C_0^\infty(X)$ there exist $C,h_0$ such that
\begin{equation}\label{e:log2}
\|\chi_0(P - \lambda)^{-1}\chi_0\| \le C \frac{\log^2(1/h)}{h}
\end{equation}
for $0<h\le h_0$ and $\re \lambda =0$, $\im \lambda \ge 0$.
\end{lem}

It is natural to conjecture that $\log^2(1/h)$  could be improved to $\log(1/h)$ in \eqref{e:log2}. This is the (optimal) bound obtained in \cite{c07,c08,Nonnenmacher-Zworski:Quantum, wz} in various settings where there is hyperbolic trapping. For example if one had the analogue of \cite[Theorem A]{bz}  or \cite[(1.6)]{c07} for the model operator $P_1$ in the proof below, the parametrix construction would give this bound

\begin{proof}
We define two model operators: a nontrapping model $P_0$ and a trapping model $P_1$. Unlike in the usual setup, the nontrapping model is ``compact'' in the sense that it agrees with $P$ only for small values of $s$, while the region near $\{|s| \ge s_0\}$ is suppressed by a complex absorbing barrier.  Meanwhile the trapping model is ``noncompact'' in the sense that it agrees with $P$ outside a small neighborhood of $\{|s| = 0\}$, and only the region near $\{|s| = 0\}$ is suppressed. For the resolvent of $P_0$ we prove an $\Oh(h^{-1})$ bound (this is standard), and for the resolvent of $P_1$ we prove an $\Oh(h^{-1}\log(h^{-1}))$ bound (for this we use \cite[Theorem~2.1]{DaVa}), after which an $\Oh(h^{-1}\log^2(h^{-1}))$ bound for $P$ follows by a slightly more complicated version of the parametrix construction of \cite[\S3]{DaVa}.

More concretely, let $W_0 \in C^\infty(X;[0,1])$ be $0$ for $|s| \le 5s_0/7$ and $1$ for $|s| \ge 6s_0/7$. Let $W_1 \in C^\infty(X;[0,1])$ be $0$ for $|s| \ge 2s_0/7$ and $1$ for $|s| \le s_0/7$.

\begin{figure}[htb]
\includegraphics{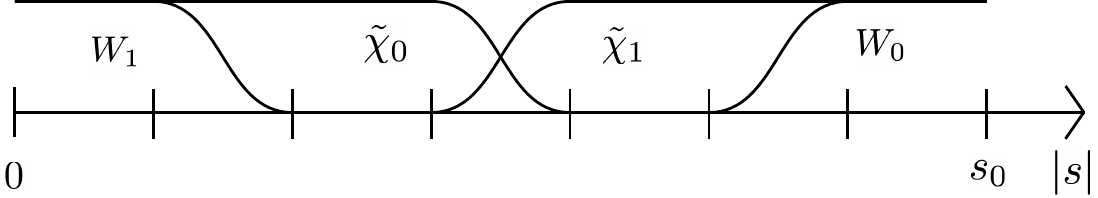}
\caption{The complex absorbing barriers and cutoffs of Lemma~\ref{l:polybd}.}\label{f:barriersandcutoffs}
\end{figure}

Then put
\[P_j \Def P - iW_j, \qquad X_j \Def X\setminus \supp W_j, \qquad P_j|_{X_j} = P|_{X_j}, \qquad j \in \{0,1\},\]
and let $R_j(\lambda) \Def (P_j-\lambda)^{-1}$. By the discussion in \S\ref{s:bar} we have
\[\|\chi_0R_0(\lambda)\chi_0\| \le C h^{-1},\]
because all backward bicharacteristics enter $\{W_0 = 1\}$.
Meanwhile the trapping in $P_1$ consists of four isolated closed hyperbolic orbits, and hence
\[\|\chi_0R_1(\lambda)\chi_0\| \le C \frac{\log(1/h)}h,\]
by \cite[Theorem 2.1]{DaVa}, where we used the convexity conditions \eqref{e:convinf} and \eqref{e:convcompact} to glue a trapping estimate near the hyperbolic orbits (such as \cite[(1.1)]{wz}) to a nontrapping estimate for the infinite end (such as \cite[(1.6)]{Cardoso-Popov-Vodev:Semiclassical}). Note that by the discussion in \S\ref{s:outgoing}, the resolvents of both $P_0$ and $P_1$ are semiclassically outgoing. Now let $\tilde \chi_0 \in C^\infty(\R;[0,1])$ be $1$ near $|s| \le 3s_0/7$ and $0$ near $|s| \ge 4s_0/7$, let $\tilde \chi_1 = 1 -\tilde \chi_0$, and define a right parametrix for $P$ by
\[F \Def \tilde\chi_0(s-s_0/7) R_0(\lambda)\tilde \chi_0(s) + \tilde \chi_1(s+s_0/7) R_1(\lambda) \tilde \chi_1(s).\]
An iterated construction using this $F$, as in \cite[\S3]{DaVa} (or as in \S\ref{s:first} above), gives \eqref{e:log2}. More specifically,  put
\[\begin{split}
(P - \lambda)F &= \Id + [P, \tilde\chi_0(s-s_0/7) ]R_0(\lambda)  \tilde \chi_0(s)+ [P,\tilde \chi_1(s+s_0/7) ]R_1(\lambda)\tilde \chi_1(s) \\&\Def \Id + A_0 + A_1.
\end{split}\]
Although we have $A_0^2 = A_1^2 = 0$ as before, $\|A_0A_1\| \ne \Oh(h^\infty)$ because there are bicharacteristics which pass from $\supp \tilde \chi_1$ to $\supp d \tilde \chi_1(s+s_0/7)$ to $\supp \tilde \chi_0(s-s_0/7)$. We accordingly iterate the parametrix three more times, writing
\[\begin{split}
(P  - \lambda) F(\Id -  A_0 -  A_1 +  A_1 A_0 +  A_0 A_1 - A_0A_1A_0 -  A_1A_0A_1) \\= \Id -  A_1A_0 A_1 A_0 - A_0A_1 A_0 A_1,
\end{split}\]
The remainder is trivial in the sense that
\[\|A_1A_0A_1A_0 \chi_0\| + \|A_0A_1A_0A_1\| = \Oh(h^\infty),\]
and our parametric obeys the estimate
\[
\|\chi_0F(\Id -  A_0 -  A_1 +  A_1 A_0 +  A_0 A_1 - A_0A_1A_0 -  A_1A_0A_1)\chi_0\| \le C \frac{\log^2(1/h)}h,
\]
completing the proof of \eqref{e:log2}.
\end{proof}

Now by the discussion in \S\ref{s:outgoing} the resolvent $(P-\lambda)^{-1}$ is semiclassically outgoing, and recall that trajectories which intersect  $\{w=1\}$ at some negative time are considered backward nontrapped, allowing us to apply Theorem 1 with $\Gamma$ taken to be one or several of the hyperbolic closed orbits at $s = \pm  s_0$. For example, we have the following statement:

\begin{prop}\label{p:contrivedexample}
For all $\chi_1 \in C_0^\infty(X)$ with $\supp \chi_1 \cap \{|s| = s_0\} = \emptyset$, there exist $C,h_0$ such that
\[\|\chi_1(P - \lambda)^{-1}\chi_1\| \le \frac C h\]
for $0<h\le h_0$ and $\re \lambda =0$, $\im \lambda \ge 0$.
\end{prop}

\begin{proof} We closely follow \S\ref{s:cor}. Let $f = \chi_1 v$, $\|v\| = \Oh(1)$, $u = (P-\lambda)^{-1}f$.  We must show that for any $\rho \in T^*\supp \chi_1$, $u$ is $\Oh(h^{-1})$ at $\rho$. There are four cases.
\begin{enumerate}
\item If $\rho \not\in \Sigma = p^{-1}(0)$, then $u$ is $\Oh(1)$ at $\rho$ by elliptic regularity and the polynomial boundedness of the resolvent.
\item If $\rho$ is backward nontrapped (i.e. either escapes to infinity or enters the interior of the support of $w$), then, because $(X,g)$ is asymptotically conic or hyperbolic, $u$ is $\Oh(h^{-1})$ at $\rho$ by the discussion in \S\ref{s:bar} and \S\ref{s:inf}.
\item If $\lim _{t \to -\infty} s(\gamma_\rho(t)) = -s_0$, then $u$ is $\Oh(h^{-1})$ at $\rho$ by Theorem~\ref{t:main} applied with $\Gamma$ the union of the two closed orbits at $s = -s_0$. The assumption in Theorem~\ref{t:main} that $u$ is $\Oh(h^{-1})$ on $\Gamma_-$ follows from case (2) above.
\item If $\lim _{t \to -\infty} s(\gamma_\rho(t)) = s_0$, then $u$ is $\Oh(h^{-1})$ at $\rho$ by Theorem~\ref{t:main} applied with $\Gamma$ the union of the two closed orbits at $s = s_0$. The assumption in Theorem~\ref{t:main} that $u$ is $\Oh(h^{-1})$ on $\Gamma_-$ follows from cases (2) and (3) above.
\end{enumerate}
This proves that
\[\| \chi_1(P - \lambda)^{-1}\chi_1v\| \le Ch^{-1}.\]
The uniformity in $v$ follows from Banach-Steinhaus.
\end{proof}

\section{Semiclassically outgoing resolvents}\label{s:outgoing}
In this section we discuss the assumption that the resolvent family is semiclassically outgoing. As mentioned above, this condition replaces any explicit assumptions about the structure of the manifold near infinity and allows us to work in an arbitrarily small neighborhood of the trapped set. In \S\ref{s:bar} we explain this condition in the case of a polynomially bounded resolvent with a complex absorbing barrier added, a convenient simple model of infinity used to study resolvents in trapping geometries. In \S\ref{s:inf} we consider manifolds which are asymptotically conic or asymptotically hyperbolic in the sense of \S\ref{s:defnot}. Finally, in \S\ref{s:3body} we give an example from $3$-body scattering, illustrating that this assumption is flexible in the sense that it can hold on a manifold whose natural compactification is a manifold with corners rather than a manifold with boundary, and which is not covered by the analysis of \cite{Cardoso-Vodev:Uniform, Cardoso-Popov-Vodev:Semiclassical}. Introducing a suitable short-range three-particle interaction in this setting can produce a hyperbolic trapped set to which Theorem \ref{t:main} can be applied.

In all the examples discussed in this section, the semiclassically outgoing condition with a quantified $h^{-1}$ loss follows from the proof of the same condition without the quantified loss. This weaker condition is discussed in \cite{DaVa} for several of the examples below, and since no significant changes are needed we omit many details.

\subsection{Complex absorbing barriers}\label{s:bar} The simplest setting in which the resolvent is semiclassically outgoing is when ``infinity is suppressed'' by a \textit{complex absorbing barrier}, which we denote by adding a term of the form $-iW$ to a Schr\"odinger operator. See \cite{Nonnenmacher-Zworski:Quantum} and \cite{wz} for examples of theorems about resolvent estimates in the presence of trapping which are simplified in this setting, and see \cite{DaVa} for a general method for gluing in another (more interesting) semiclassically outgoing infinity once such an estimate is proved. This method is used in the present paper to construct the example in \S\ref{s:app}.

The following lemma is standard, and the proof is essentially the same as that of \cite[Lemma 5.1]{DaVa}.

\begin{lem}
Let $(X,g)$ be a complete Riemannian manifold, let $P_0 = h^2\Delta_g + V$ be a semiclassical Schr\"odinger operator with $V \in C^\infty(X)$, let $P = P_0 - iW$ where $W = \Op(w)$ and $w \in C^\infty(T^*X;[0,1])$ is identically $1$ off a compact subset of $T^*X$, and let $I \subset \R$ be compact. Suppose $R_h(\lambda) \Def (P-\lambda)^{-1}$, is polynomially bounded for $\lambda \in D \subset \{\re \lambda \in I, \im \lambda \ge -\Oh(h^\infty)\}$. Then $R_h(\lambda)$ is semiclassically outgoing for $\lambda \in D$.
\end{lem}

In applications, $w$ is often chosen to be identically $0$ near $\Gamma$, and the assumption on $w$  is often replaced by the assumption that $w \in C^\infty(X;[0,1])$ with $w$ identically $1$ off a compact subset of $X$.

\subsection{Asymptotically conic and hyperbolic manifolds}\label{s:inf}
On an asymptotically conic manifold (see \S\ref{s:defnot} for a definition), the semiclassically outgoing assumption  follows from the construction and estimates of \cite{vz}: see \cite[Lemma 2]{d} for a very similar statement.  On an even asymptotically hyperbolic manifold (see \S\ref{s:defnot} for a definition) the semiclassically outgoing property is proved in \cite[Theorem 4.3]{v1} (see also \cite[Theorem 5.1]{v2}).

Another approach is possible in the case when $(X,g)$ is asymptotically hyperbolic and satisfies the additional assumptions that each connected component of $\D \overline{X}$ is a sphere and that
\[g = g_\HH + \tilde g,\]
where $g_\HH$ is a symmetric cotensor which agrees with the hyperbolic metric on $\HH^n$ in a neighborhood of each connected component $\D \overline{X}$, and $\tilde g$ is a symmetric cotensor smooth up to $\D \overline{X}$.  Namely, one can use an argument similar to that in \cite[\S4.2]{DaVa} and derive the semiclassically outgoing property from a description given in \cite{Melrose-SaBarreto-Vasy:Semiclassical} of the Schwartz kernel of the resolvent as a paired Lagrangian distribution, to which a semiclassical version of \cite[Theorem 3.3]{gu} can be applied. 

\subsection{An example from $3$-body scattering}\label{s:3body} Consider the following $3$-body Hamiltonian on $\R^3$:
\[P_0 = -h^2 \D_{x_1}^2 - h^2  \D_{x_2}^2 - h^2  \D_{x_3}^2 + V(x_1-x_2) + V(x_2-x_3) + V(x_3-x_1) - 1,\]
where $V \in C_0^\infty(\R)$. The particles here are constrained to move on a line, and $V$ is the interaction potential between each pair of them.  Passing to center of mass coordinates, we obtain the following reduced Hamiltonian on the plane $X = \{x_1 + x_2 + x_3 = 0\}$:
\[P = -h^2 \Delta + \pi_1^*V + \pi_2^*V + \pi_3^*V - 1,\]
where $\pi_1$ is the projection $(x_1,x_2,x_3)\mapsto x_1 - x_2$, and similarly for $\pi_2$ and $\pi_3$. Note that even when $V$ is small, the perturbation is very long range (and consequently not covered by \cite{Cardoso-Vodev:Uniform, Cardoso-Popov-Vodev:Semiclassical}), and it cannot be extended smoothly to a compactification $\overline{X}$ of $X$ unless $\overline{X}$ is a manifold with corners.

In \cite{Gerard:Semiclassical}, G\'erard shows that if $V$ is classically nontrapping (for example it suffices to take $V$ small) then the resolvent obeys the standard nontrapping bound:
\[\|\chi R_h(\lambda) \chi \| \le C h^{-1},\]
for $0<h\le h_0$, $|\re \lambda| \le \eps_0<1$, $\im \lambda \ge 0$. Moreover the methods of the paper, more explicitly elaborated by Wang \cite{Wang:Semiclassical} in the general $N$-body setting, imply that the resolvent is semiclassically outgoing \cite[(1.8)]{Wang:Semiclassical}.  More specifically, Wang shows that for $f$ compactly supported, $R_h(\lambda) f$ is $\Oh(h^\infty)$ near spatial infinity, where the radial momentum is negative. If $V$ is nontrapping, any backward bicharacteristic eventually enters this region, so the semiclassically outgoing condition follows from propagation of singularities.


\begin{thebibliography}{00}
\bibitem[BBR10]{Bony-Burq-Ramond} Jean-Fran\c cois Bony, Nicolas Burq and Thierry Ramond. Minoration de la r\'esolvante dans le cas captif. [Lower bound on the resolvent for trapped situations]. \textit{C. R. Math. Acad. Sci. Paris.} 348(23-24):1279--1282, 2010.

\bibitem[BoPe06]{Bony-Petkov:Resolvent} Jean-Fran\c cois Bony and Vesselin Petkov. Resolvent estimates and local energy decay for hyperbolic equations. \textit{Ann. Univ. Ferrara Sez. VII Sci. Mat.}, 52(2):233--246, 2006.

\bibitem[Bur02]{Burq:Lower}
Nicolas Burq.
\newblock Lower bounds for shape resonances widths of long range
  {S}chr\"odinger operators.
\newblock {\em Amer. J. Math.}, 124(4):677--735, 2002.

\bibitem[BGH10]{bgh} Nicolas Burq, Colin Guillarmou and Andrew Hassell. Strichartz estimates without loss on manifolds with hyperbolic trapped geodesics. \textit{Geom. Funct. Anal.}, 20:627--656, 2010.

\bibitem[BuZw04]{bz} Nicolas  Burq and Maciej Zworski. Geometric control in the presence of a black box. \textit{J. Amer. Math. Soc.} 17:2, 443--471, 2004.

\bibitem[CPV04]{Cardoso-Popov-Vodev:Semiclassical} Fernando Cardoso, Georgi Popov and Georgi Vodev. Semi-classical resolvent estimates for the Schr\"odinger operator on non-compact complete Riemannian manifolds. \textit{Bull. Braz. Math. Soc. (N.S.)} , 35(3):333-344, 2004.

\bibitem[CaVo02]{Cardoso-Vodev:Uniform}
Fernando Cardoso and Georgi Vodev.
\newblock Uniform estimates of the resolvent of the {L}aplace-{B}eltrami
  operator on infinite volume {R}iemannian manifolds. {II}.
\newblock {\em Ann. Henri Poincar\'e}, 3(4):673--691, 2002.

\bibitem[Chr07]{c07} Hans Christianson. Semiclassical non-concentration near hyperbolic orbits. \textit{J. Funct. Anal.} 246(2):145--195, 2007.

\bibitem[Chr08]{c08}Hans Christianson. Dispersive estimates for manifolds with one trapped orbit. \textit{Comm. Partial Differential Equations}, 33(7):1147--1174, 2008.

\bibitem[ChWu11]{cw} Hans Christianson and Jared Wunsch. Local smoothing for the Schr\"odinger equation with a prescribed loss. Preprint available at arXiv:1103.3908.

\bibitem[Dat09]{d} Kiril Datchev. Local smoothing for scattering manifolds with hyperbolic trapped sets. \textit{Comm. Math. Phys.} 286(3):837--850, 2009.


\bibitem[DaVa10]{DaVa} Kiril Datchev and Andr\'as Vasy. Gluing semiclassical resolvent estimates via propagation of singularities. Preprint available at arXiv:1008.3964, 2010.

\bibitem[DeG\'e99]{Derezinski-Gerard} Jan Derezi\'nski and Christian G\'erard. Scattering theory of classical and quantum $N$-particle systems. \textit{Texts and Monographs in Physics.} Springer-Verlag, Berlin, 1997.

\bibitem[DiSj99]{ds} Mouez Dimassi and Johannes Sj\"ostrand. Spectral asymptotics in the semiclassical limit. \textit{London Math. Soc. Lecture Note Ser.} 268, 1999.

\bibitem[EvZw10]{ez} Lawrence C. Evans and Maciej Zworski. Semiclassical analysis. Lecture notes available online at \url{http://math.berkeley.edu/~zworski/semiclassical.pdf}.

\bibitem[G\'er90]{Gerard:Semiclassical} Chrisitan G\'erard. Semiclassical resolvent estimates for two and three-body Schršdinger operators. \textit{Comm. Partial Differential Equations}15(8):1161--1178, 1990.

\bibitem[G\'eSj87]{GeSj} Christian G\'{e}rard and Johannes Sj\"{o}strand. Semiclassical
resonances generated by a closed trajectory of hyperbolic type. \textit{Comm. Math. Phys.} 108:391--421, 1987.


\bibitem[GrUh90]{gu} Allan Greenleaf and Gunther Uhlmann. Estimates for singular radon transforms and pseudodifferential operators with singular symbols. \textit{J. Func. Anal.}  89(1):202--232, 1990.

\bibitem[Gui05]{g} Colin Guillarmou. Meromorphic properties of the resolvent on asymptotically hyperbolic manifolds. \textit{Duke Math. J.} 129(1):1--37, 2005.

\bibitem[H\"o94]{h} Lars~H\"ormander, {\em The Analysis of               
Linear Partial Differential Operators. III. Pseudo-Differential Operators,\/}
Springer Verlag, 1994.



\bibitem[MSV11]{Melrose-SaBarreto-Vasy:Semiclassical}
Richard B. Melrose, Ant\^onio S{\'a} Barreto, and Andr\'as Vasy.
\newblock Analytic continuation and semiclassical resolvent estimates on
  asymptotically hyperbolic spaces.
\newblock Preprint available at arXiv:1103.3507, 2011.


\bibitem[NoZw09]{Nonnenmacher-Zworski:Quantum}
St{\'e}phane Nonnenmacher and Maciej Zworski.
\newblock Quantum decay rates in chaotic scattering.
\newblock {\em Acta Math.}, 203(2):149--233, 2009.

\bibitem[PeSt09]{Petkov-Stoyanov:Singularities} Vesselin Petkov and Luchezar Stoyanov, Singularities of the scattering kernel related to trapping rays. In \textit{Advances in phase space analysis of partial differential equations}, volume 78 of \textit{Progr. Nonlinear Differential Equations Appl.}, 235--251, 2009.

\bibitem[Ral71] {ralston} James V. Ralston. Trapped rays in spherically symmetric media and poles of the scattering matrix. \textit{Comm. Pure Appl. Math.} 24, 571--582, 1971.


\bibitem[SiSo87]{Sigal-Soffer:N} Israel~M. Sigal and Avy~Soffer. \newblock $N$-particle scattering problem: asymptotic completeness for short  range systems. \newblock {\em Ann. Math.}, 125:35--108, 1987.


\bibitem[Vas03]{Vasy:Geometry}
Andr{\'a}s Vasy.
\newblock Geometry and analysis in many-body scattering.
\newblock In {\em Inside out: inverse problems and applications}, volume~47 of
  {\em Math. Sci. Res. Inst. Publ.},  333--379, 2003.
  
  \bibitem[Vas10]{v1} Andr\'as Vasy. Microlocal analysis of asymptotically hyperbolic and Kerr-de Sitter spaces. Preprint available at arXiv: 1012.4391, 2010.

\bibitem[Vas11]{v2} Andr\'as Vasy. Microlocal analysis of asymptotically hyperbolic spaces and high energy resolvent estimates. Preprint available at arXiv:1104.1376, 2011.

\bibitem[VaZw00]{vz} Andr\'as Vasy and Maciej Zworski. Semiclassical estimates in asymptotically Euclidean scattering. \textit{Comm. Math. Phys.} 212(1):205--217, 2000.
 
 \bibitem[Wan91]{Wang:Semiclassical} Xue Ping Wang. Semiclassical resolvent estimates for $N$-body Schršdinger operators. \textit{J. Funct. Anal.} 97(2):466--483, 1991.

\bibitem[WuZw10]{wz} Jared Wunsch and Maciej Zworski. Resolvent estimates for normally hyperbolic trapped sets. \textit{Ann. Inst. Henri Poincar\'e (A).} 12(7):1349--1385, 2011.


\end{thebibliography}
\end{document}